\documentclass[12pt,a4paper]{amsart}

\usepackage[margin=1.3 in]{geometry}

\numberwithin{equation}{section}

\usepackage{amsmath} 
\usepackage{amsthm} 
\usepackage{amsfonts} 
\usepackage{amssymb}

\usepackage{tikz}

\usepackage{longtable}
\usepackage{caption}
\usepackage{array}

\newcommand\mf {\mathfrak }
\newcommand\mc{\mathcal}
\newcommand\wh{\widehat}
\newcommand\wt{\widetilde}
\newcommand\ab{\mathrm{ab}}

\newcommand\al{\alpha}
\newcommand\be{\beta}
\newcommand\ga{\gamma}
\newcommand\de{\delta}

\newcommand\ganz{\mathbb Z} 
\newcommand\real{\mathbb R}

\newcommand\conv{\mathrm{Conv}}

\newcommand\codim{\mathrm{codim}\,}
\newcommand\stab{\mathrm{Stab}}
\newcommand\vol{\mathrm{Vol}}
\newcommand\gen{\mathrm{Span}}
\newcommand\supp{\mathrm{Supp}}

\renewcommand\){\right)}

\newtheorem*{thm*}{Theorem}
\newtheorem{thm}{Theorem}[section]
\newtheorem{pro}[thm]{Proposition}
\newtheorem{lem}[thm]{Lemma}
\newtheorem{cor}[thm]{Corollary}
\newtheorem*{cor*}{Corollary}

\renewcommand\ss{\scriptstyle}
\newcommand\la{\langle}
\newcommand\ra{\rangle}
\newcommand\ov{\overline}
\newcommand\un{\underline}

\theoremstyle{definition}
\theoremstyle{definition}\newtheorem{rem}{Remark}[section]
\theoremstyle{definition}
\theoremstyle{definition}\newtheorem{exa}{Example}[section]

\newcommand\Anfi
{\draw{(0,0)--(1.5,0)
           (2.5,0)--(3.5,0)
           (6.5,0)--(8,0)};
\draw[dashed]{ (1.5,0)--(2.5,0)(3.5,0)--(6.5,0)};
\foreach \x in {0, 1, 3, 7, 8}
\draw[fill=white]{(\x,0) circle(3pt)};}

\newcommand\An
{\draw{(0,0)--(4,1)--(8,0)};
\draw[fill=black]{(4,1) circle(3pt)};
\Anfi}

\newcommand\Bnf
{\draw{(1,0)--(2,0)--(2.5,0) 
 (5.5,0)--(6,0)--(7,0)};
\draw[dashed]{ (2.5,0)--(5.5,0)};
\draw{(7,-1.5pt)--(8,-1.5pt)};
\draw{(7,1.5pt)--(8,1.5pt)};
\draw{(7.4,.15)--(7.55,0)--(7.4,-.15)};
\foreach \x in {1,2,6,7,8}
\draw[fill=white]{(\x,0) circle(3pt)};}

\newcommand\Bn
{\draw{ (.3,-.7)--(1,0)--(.3,.7)};
\draw[fill=white]{(.3,.7) circle(3pt)};
\draw[fill=black]{(.3,-.7) circle(3pt)};
\Bnf}

\newcommand\Cnf
{\draw{(1,0)--(2,0)--(2.5,0) (5.5,0)--(6,0)--(7,0)};
\draw[dashed]{ (2.5,0)--(5.5,0)};
\draw{(7,-1.5pt)--(8,-1.5pt)};
\draw{(7,1.5pt)--(8,1.5pt)};
\draw{(7.6,.15)--(7.45,0)--(7.6,-.15)};
\foreach \x in {1,2,6,7,8}
\draw[fill=white]{(\x,0) circle(3pt)};}

\newcommand\Cn
{\draw{(0,-1.5pt)--(1,-1.5pt)};
\draw{(0,1.5pt)--(1,1.5pt)};
\draw{(.4,.15)--(.55,0)--(.4,-.15)};
\draw[fill=black]{(0,0) circle(3pt)};
\Cnf}

\newcommand\Dnf
{\draw{(1,0)--(2,0)--(2.5,0) 
(5.5,0)--(6,0)--(7,0)--(7.7,.7) (7,0)--(7.7,-.7)};
\draw[dashed]{ (2.5,0)--(5.5,0)};
\foreach \x in {1,2,6,7}
\draw[fill=white]{(\x,0) circle(3pt)};
\draw[fill=white]{(7.7,.7) circle(3pt)};
\draw[fill=white]{(7.7,-.7) circle(3pt)};}

\newcommand\Dn
{\draw{ (.3,-.7)--(1,0)--(.3,.7)};
\draw[fill=white]{(.3,.7) circle(3pt)};
\draw[fill=black]{(.3,-.7) circle(3pt)};
\Dnf}

\newcommand\Eseif
{\foreach \x in {0, 1, 2,3}
\draw{(\x,1)--(\x+1,1)};
\draw{(2,1)--(2,0)};
\foreach \x in {0, 1, 2,3,4}
\draw[fill=white]{(\x,1) circle(3pt)};
\draw[fill=white]{(2,0) circle(3pt)};}

\newcommand\Esei
{\draw{(2,0)--(2,-1)};
\draw[fill=black]{(2,-1) circle(3pt)};
\Eseif}

\newcommand\Esettef
{\foreach \x in {1, 2,3,4,5}
\draw{(\x,0)--(\x+1,0)};
\draw{(3,0)--(3,-1)};
\foreach \x in {1, 2,3,4,5,6}
\draw[fill=white]{(\x,0) circle(3pt)};
\draw[fill=white]{(3,-1) circle(3pt)};}

\newcommand\Esette
{\draw{(0,0)--(1,0)};
\draw[fill=black]{(0,0) circle(3pt)};
\Esettef}

\newcommand\Eottof
{\foreach \x in {0,1,2,3,4,5}
\draw{(\x,0)--(\x+1,0)};
\draw{(2,0)--(2,-1)};
\foreach \x in {0,1,2,3,4,5,6}
\draw[fill=white]{(\x,0) circle(3pt)};
\draw[fill=white]{(2,-1) circle(3pt)};}

\newcommand\Eotto
{\draw{(6,0)--(7,0)};
\draw[fill=black]{(7,0) circle(3pt)};
\Eottof}

\newcommand\Fquattrof
{\foreach \x in {1,3}
\draw{(\x,0)--(\x+1,0)};
\draw{(2,1.5pt)--(3,1.5pt)};
\draw{(2,-1.5pt)--(3,-1.5pt)};
\draw{(2.4,.15)--(2.55,0)--(2.4,-.15)};
\foreach \x in {1,2,3,4}
\draw[fill=white]{(\x,0) circle(3pt)};}
\newcommand\Fquattro
{\draw{(0,0)--(1,0)};
\draw[fill=black]{(0,0) circle(3pt)};
\Fquattrof}

\newcommand\Gduef
{\foreach\y in{0,2,-2}
\draw{(0,\y pt)--(1,\y pt)};
\draw{(.6,.15)--(.45,0)--(.6,-.15)};
\foreach \x in {0,1}
\draw[fill=white]{(\x,0) circle(3pt)};}

\newcommand\Gdue
{\draw{(1,0)--(2,0)};
\draw[fill=black]{(2,0) circle(3pt)};
\Gduef}

\title{Root polytopes and abelian ideals}

\author{Paola Cellini}
\address{Paola Cellini\\ Dipartimento di Ingegneria e Geologia\\ Universit\`a di Chieti -- Pescara\\ Viale Pindaro 42\\ 65127 Pescara\\ Italy}
\email{pcellini@unich.it}

\author{Mario Marietti}
\address{Mario Marietti\\ Dipartimento di Ingegneria Industriale e Scienze Matematiche\\ Universit\`a  Politecnica delle Marche\\ Via Brecce Bianche\\ 60131 Ancona\\ Italy}
\email{m.marietti@univpm.it}

\begin{document}

\begin{abstract}
We study the root polytope $\mathcal P_\Phi$ of a finite irreducible crystallographic root system $\Phi$ using its relation with the abelian ideals of a Borel subalgebra of a simple Lie algebra with root system $\Phi$. We determine the hyperplane arrangement corresponding to the faces of codimension 2 of $\mathcal P_\Phi$ and analyze its relation with the facets of $\mathcal P_\Phi$. For $\Phi$ of type $A_n$ or $C_n$, we show that the orbits of some special subsets of abelian ideals under the action of the Weyl group parametrize a triangulation of $\mathcal P_\Phi$. We show that this triangulation restricts to a triangulation of the positive root polytope $\mathcal P_\Phi^+$. 
\end{abstract}

\maketitle

\setlength{\baselineskip}{1.2\baselineskip}

\section{Introduction}
 Let $\Phi$ be a finite irreducible crystallographic root system in a Euclidean space, $\Pi$ a basis of $\Phi$, and $\Phi^+$ the corresponding set of positive roots. We denote by $\mc P_\Phi$, or simply by $\mc P$, the convex hull of $\Phi$ and call $\mc P$ the root polytope of $\Phi$. Moreover, we denote by  $\mc P^+_\Phi$, or  simply by $\mc P^+$, the convex hull of $\Phi^+\cup\{\un 0\}$, where $\un 0$ is the zero vector, and call $\mc P^+$ the positive root polytope of $\Phi$. 
The root polytope, the positive root polytope, and their triangulations have been studied by several authors such as in \cite{ABHPS}, \cite{GGP}, \cite{K1},
\cite{K2},  for some or all the classical root systems. In \cite{C-M}, we have
given a case free description of $\mc P$ for arbitrary $\Phi$. In particular,  we obtained a uniform explicit description of its faces, its $f$-polynomial, and a minimal set of linear inequalities that defines $\mc P$ as an intersection of half-spaces.  In  this paper, we develop our algebraic-combinatorial analysis of $\mc P$, mostly for
the types $A_n$ and $C_n$, taking into account also the positive root polytope. Our
analysis relies on the results of our previous paper which we
briefly recall. 
\par

Let $\mf  g$  be a complex simple Lie algebra with Cartan subalgebra $\mf h$ and corresponding root system $\Phi$, $\mf g_\al$ the root space of $\al$,  for all $\al\in \Phi$, and $\mf b$ the Borel subalgebra of $\mf g$ corresponding to $\Phi^+$. 
Moreover, let $W$ be the Weyl group of $\Phi$. It is clear that $W$ acts on the set of the faces of $\mc P$.  
In \cite{C-M} we showed that there is a natural bijection between the set of the orbits of this action and a certain set of abelian ideals of $\mf b$. 
\par

The abelian ideals of a Borel subalgebra of a complex simple Lie algebra are a
well studied subject. 
If $\mf a$ is a nontrivial abelian ideal of $\mf b$, then there exists a subset
$I_\mf a$ of $\Phi^+$  such that  $\mf a=\bigoplus_{\al\in I_\mf a} \mf g_\al$. 
If $I$ is a subset of $\Phi^+$, then $\bigoplus_{\al\in I} \mf g_\al$ is an abelian ideal of $\mf b$ if and only if $I$ is a filter  (or dual order ideal) of the poset $\Phi^+$ with respect to the usual order of the root lattice, and moreover, the sum of any two roots in $I$ is not a root.
We call a subset of $\Phi^+$ with these properties, by abuse of language, an abelian ideal of $\Phi^+$.
\par

In \cite{C-M}, we proved that each orbit of the action of $W$ on the set of the faces of $\mc P$ contains a distinguished representative, 
which is the convex hull of an abelian ideal of $\Phi^+$, and we determined
explicitly the abelian ideals corresponding to these faces of $\mc P$. 
Moreover, we proved that these abelian ideals are all principal i.e., of type
$M_\al:=\{\be\in \Phi^+\mid \be\geq \al\}$, for some $\al\in \Phi^+$. 
We call the faces corresponding to these abelian ideals the standard parabolic
faces. 
\par

In this work we continue our analysis of $\mc P$ in two independent directions: we study a hyperplane arrangement associated with the root polytope of any irreducible root system, and certain triangulations of type $A$ and $C$ root polytopes and positive root polytopes. 
The common point is the good behavior of the two types $A_n$ and $C_n$, with respect to the other types. 
\par

First, in Sections \ref{a hyp} and \ref{ac-si},  we analyze the central arrangement $\mc H_\Phi$ of all linear hyperplanes containing some codimension 2 face of $\mc P_\Phi$. 
Using the explicit description of the faces of $\mc P_\Phi$ given in \cite{C-M},
we determine the hyperplanes in $\mc H_\Phi$ and their 
orbit structure under the action of $W$, for all $\Phi$.  
\par

By construction, for each facet $F$ of $\mc P_\Phi$, the cone generated by $F$ is the closure of a union of regions of $\mc H_\Phi$.  
We show that, for the types $A_n$ and $C_n$, each cone on a facet is the closure of a single region of $\mc H_\Phi$, so that there is a natural bijection between the regions of $\mc H_\Phi$ and the facets of $\mc P_\Phi$. 
The analogous result does not hold for the types $B_n$ and $D_n$. 
 Moreover, in both the cases $A_n$ and $C_n$, $\mc H_\Phi$ is the orbit of $\omega_1^\perp$  under the action of $W$, where $\omega_1$ is the first fundamental weight.  
\par

In the rest of the paper we deal with our second topic, the triangulations of
the type $A_n$ and $C_n$ root polytopes and positive root polytopes. The main
special property of these types is that, in these cases, the abelian ideals
corresponding to the standard parabolic facets of $\mc P$ are exactly the
maximal abelian ideals in $\Phi^+$. 
\par

The poset of the abelian ideals of $\Phi^+$ with respect to inclusion is quite well known (see \cite{Pany}, \cite{Sut} and \cite{C-P-Adv} for results and motivations). 
In the study of it, a key role is played by some special subposets $\mc I_\ab(\al)$ (see Section \ref{ideals}), $\alpha$ any long simple root.
The poset structure of the sets $\mc I_\ab(\al)$ is well known. In particular, each $\mc I_\ab(\al)$ has a maximum $I(\al)$,  and the set of all $I(\al)$, with $\alpha$ long simple root, is the set of all maximal abelian ideals of $\Phi^+$.
For $\Phi$ of type $A_n$ or $C_n$,  for all long simple $\alpha$, the maximum $I(\al)$ is the principal ideal $M_\al$. By the results of \cite{C-M}, for these types, the standard parabolic facets are the convex hulls of these same $M_\al$, hence they naturally correspond to the maximal abelian ideals.
\par

One of our main results is that, for both the types  $A_n$ and $C_n$, the set $\mc I_\ab(\al)$ parametrizes in a natural way a triangulation of the standard parabolic facet corresponding to $M_\al$  (Theorem \ref{tricomune}). 
In both cases, to each abelian ideal $\mf i$ in $\mc I_\ab(\al)$,  we associate
a simplex $\sigma_\mf i$, in such a way that the resulting set of simplices
yields a triangulation of the standard parabolic facet corresponding to $M_\al$.
For each $\mf i$, the vertices of $\sigma_{\mf i}$ form a $\ganz$-basis of the root lattice. 
Since the short roots are never vertices of $\mc P$, in type $C_n$ there are vertices of $\sigma_{\mf i}$ that are not vertices of $\mc P$. 
The resulting triangulation inherits, in a certain sense, the poset structure of
$\mc I_\ab(\al)$. 
In fact, if $\mf j\in \mc I_\ab(\al)$, then the set of all simplices $\sigma_\mf i$ with $\mf i\subseteq \mf j$ yields a triangulation of the convex hull of $\mf j$. 
\par

If $F$ is a standard parabolic facet and $W_F$ is its stabilizer in $W$, we can
extend the triangulation of $F$ to its orbit by means of a system of
representatives of the set of left cosets  $W/W_F$. In this way, from our 
triangulations of the facets, we obtain a triangulation of the boundary of $\mc
P$.  If we choose, for all the facets,  the system of minimal length
representatives, we obtain the easiest triangulation, from a combinatorial point
of view. For $A_n$, this is the triangulation described, with a combinatorial
approach, in \cite{ABHPS}. 
\par

 If we  extend each $(n-1)$-simplex of this triangulation of the boundary of $\mc P$ to a $n$-simplex by adding the vertex $\un 0$, we obtain a triangulation of $\mc P$.   
We prove that this triangulation of $\mc P$ restricts to a triangulation of the positive root polytope $\mc P^+$ (Theorem \ref{indottagen}). 
This implies in particular that, for $\Phi$ of type $A_n$ or $C_n$, $\mc P^+$
is the intersection of $\mc P$ with the positive cone generated by the positive
roots, which is false in general. Our result has a direct application in the recent paper 
\cite{Ch} on partition functions.
\par

For type $A_n$, the triangulation obtained for $\mc P^+$ is the one 
corresponding to the antistandard bases of \cite{GGP}. For $C_n$, our triangulation of $\mc P^+$ does not coincide with the one of \cite{K2}, though the restrictions of both triangulations on the (unique) standard parabolic facet coincide. 
\par

From our construction, we also obtain simple proofs of known and less known results about the volumes of $\mc P$ and $\mc P^+$. In particular, we obtain an analogue of the curious identity of \cite{DC-P}.
\par

Our definitions and results on triangulations, as well as the results on the
relation between $\mc H_\Phi$ and $\mc P_\Phi$, work uniformly for $\Phi$ of
type $A_n$ or $C_n$. However, we have distinct proofs for the two cases. We
sometimes illustrate our proofs and results using the encoding of the poset
$\Phi^+$ as a Young (skew for type $C$) diagram \cite{Shi} that we recall in
Section \ref{ac-si}.
\par

The content of this paper is organized as follows. In Section \ref{pre}, we fix the notation and
recall the results that we most frequently use in the paper.
In Section \ref{a hyp}, we determine the hyperplane arrangement $\mc H_{\Phi}$
for all irreducible root systems $\Phi$, and in Section \ref{ac-si} we analyze the relation between $\mc H_{\Phi}$ and   $\mc P_{\Phi}$ for the classical types.
In Section \ref{idealiabeliani}, after some general preliminary results on the principal  dual order ideals of $\Phi$, we state our main results on triangulations of $\mc P_\Phi$ and $\mc P_\Phi^+$ for the types $A_n$ and $C_n$.
The proofs of these results are given in Section \ref{tr-a} for type $A_n$ and in Section \ref{tr-c} for type $C_n$.
In Section \ref{digraphs}, we show how the simplices of the type $A_n$ triangulation
studied in Section \ref{tr-a} can be interpreted as directed graphs.
In Section \ref{volume}, we study the volumes of $\mc P_\Phi$ and $\mc P_\Phi^+$, for $\Phi$ of  type $A_n$ and $C_n$.

\section{Preliminaries}
\label{pre}
In this section, we fix the notation and recall the known results that we most  frequently use in the paper.
\par

For $n,m\in {\mathbb Z}$, with $n\le m$, we set $[n,m]=\{n,n+1,\dots,m\}$ and,
for $n\in {\mathbb N}\setminus \{0\}$, $[n]=[1,n]$. 
For every set $I$, we denote its cardinality by $|I|$.
\par

For basic facts about root systems, Weyl groups, Lie algebras, and convex
polytopes,  we refer the reader, respectively, to \cite{Bou}, \cite{BB, Hum}, \cite{H}, and~\cite{G2}. 
For any finite irreducible  crystallographic root system, we number the simple roots according to \cite{Bou}.
For the affine root system associated to $\Phi$,  we adopt the notation and definitions of \cite[Chapter 6]{Kac}, unless otherwise specified. 
\smallskip

\subsection{Root systems.}\label{rootsystems}

Let $\Phi$ be a finite irreducible (reduced) crystallographic root system in a Euclidean space $E$ with the positive definite bilinear form $(-,-)$.
\par

For all $X\subseteq E$, we denote by $\gen\,X$  the real vector space generated by $X$. Thus $E=\gen\,\Phi$. 
\par

We fix our further notation on the root system and its Weyl group in the
following list:
\smallskip 

{\renewcommand{\arraystretch}{1.2}
\begin{longtable*}{>{$} l <{$}>{$} l <{$}}
n &  \textrm{the rank of $\Phi$}, 
\\
\Pi= \{\al_1, \ldots, \al_n\} &  \textrm{set of simple roots}, 
\\
\Pi_\ell & \textrm{the set of long simple roots}, 
\\
\breve{\omega}_1, \dots, \breve{\omega}_{n} &  \textrm{the 
fundamental coweights (the dual basis of $\Pi$)},
\\
\Phi^+  &  \textrm{the set of positive roots w.r.t. $\Pi$},
\\
\Phi(\Gamma) &  \textrm{the root subsystem generated by $\Gamma$, for all
$\Gamma\subseteq \Phi$},
\\
\Phi^+(\Gamma)  &  = \Phi(\Gamma) \cap \Phi^+, 
\\
\ov{\Phi'}  & = \gen\, \Phi' \cap \Phi, \textrm{the parabolic closure of the
subsystem $\Phi'$},\\
c_\al(\be), c_i(\be) &  \textrm{the coordinates of } \be\in \Phi \textrm{ w.r.t.
} \Pi: \be=\sum_{\al\in \Pi}c_\al(\be)\al,\\
&c_i(\be)=c_{\al_i}(\be),
\\
\supp(\be) & =\{\al\in \Pi \mid c_\al(\be) \neq 0  \},
\\
\theta  & \textrm{the highest root},
\\
m_\al, m_i &   m_\al=c_\al(\theta),\  m_i=m_{\al_i},  
\\
W   &  \textrm{the Weyl group of $\Phi$},
\\
s_{\al}   & \textrm{the reflection through the hyperplane $\al^{\perp}$},
\\
\ell  &  \textrm{the length function of $W$ w.r.t. $\{s_\al\mid\al\in \Pi\}$},
\\
w_0 & \textrm{the longest element of $W$},\\
D_r(w) & =\{\al\in \Pi \mid \ell(ws_{\al})<\ell(w)\},  \textrm{ the set of right
descents of $w$},
\\
N(w)  & = \{\be \in \Phi^+ \mid w^{-1}(\be) \in -\Phi^+\},
\\
\ov{N}(w)  & = \{\be \in \Phi^+ \mid w(\be) \in -\Phi^+\}, 
\\
W\la \Gamma\ra & \textrm{the subgroup of $W$ generated by 
$\{s_\al\mid \al\in \Gamma\}$ ($\Gamma\subseteq \Phi$)}, 
\\
\wh{\Phi}  &  \textrm{the affine root system associated with $\Phi$}, 
\\
\al_0  &  \textrm{the affine simple root of $\wh{\Phi}$},
\\ 
\wh{\Pi} & = \Pi \cup \{\al_0\}, 
\\
 \wh{\Phi}^+  &  \textrm{the set of positive roots of $\wh{\Phi}$
w.r.t.  $\wh{\Pi}$}, 
\\
\wh{W} &  \textrm{the affine Weyl group of ${\Phi}$}.
\end{longtable*}
\bigskip

We call {\it integral basis} a basis of the vector space   $\gen\,\Phi$ that also is a $\ganz$-basis of the root lattice $\sum_{i\in [n]}\ganz \al_i$. 
\par

We denote by $\leq$ the usual order of the root lattice:
$\al \leq \be$ if and only if $\be - \al$ is a nonnegative linear combination of roots in $\Phi^+$. 
\par

\subsection{The sets $N(w)$}\label{Nw}
We call a set $N$ of positive roots {\em convex} if all roots that are a positive linear combination of roots in $N$ belong to $N$.  
It is clear that, for all $w$ in the Weyl group $W$, the set
$$
N(w)=\{\al\in \Phi^+\mid w^{-1}(\al)<0\}
$$
and its complement $\Phi^+\setminus N(w)$ are both convex.
\par

A set $N$ of positive roots is called {\em closed} if all roots that are a sum of two  roots in $N$ belong to $N$. It is clear that a convex set of roots is closed.
\par

If $N$ and $\Phi^+\setminus N$ are both closed, then there exists $w\in W$ such that $N=N(w)$ (\cite{Papi}, see also \cite{Pilk}). Hence, for all $N\subseteq \Phi^+$, the following three conditions are equivalent:
\begin{enumerate}
\item $N$ and $\Phi^+\setminus N$ are closed, 
\item  there exists $w\in W$ such that $N=N(w)$, 
\item $N$ and $\Phi^+\setminus N$ are convex.
\end{enumerate}
The implication (1) $\Rightarrow$ (2) holds for finite reduced crystallographic root systems and,  with some restrictions, for affine crystallographic root systems.  
The equivalence  (2) $\Leftrightarrow$  (3) holds, indeed,  for the root  system of any finitely generated Coxeter groups, provided $N$ is finite (see \cite{Pilk}).
\par

For all $w\in W$, the convexity of $\Phi^+\setminus N(w)$ implies, in particular, that  for all $\be\in N(w)$
\begin{equation}\label{supporto}
\supp(\be)\cap N(w)\neq\emptyset.
\end{equation}

We recall the following facts, which hold for the root system of any finitely generated  Coxeter group \cite[Chapter 5]{Hum}. 

If $w= s_{\al_{i_1}} \cdots s_{\al_{i_r}}$ is a reduced expression of $w \in W$, then  
\begin{equation}\label{N}
N(w)=\{\ga_1, \ldots, \ga_r\},\quad \text{where}\quad
\ga_h= s_{\al_{i_1}} \cdots s_{\al_{i_{h-1}}}(\al_{i_h})\quad\text{for}\quad  h \in [r]. 
\end{equation}
In particular, $\ell(w)=|N(w)|$. 
\par
By definition, $\ov N(w)=N(w^{-1})$, hence $\ov N(w)=\{\be_1, \ldots, \be_r\}$,
where $\be_h= s_{\al_{i_r}} \cdots$
$s_{\al_{i_{h+1}}}(\al_{i_h})$, for $h \in [r]$. 
\par

From the definition of $N$,   we directly obtain that for  $w, x\in W$, 
$$N(wx)=\(N(w)\setminus (-wN(x)\cap \Phi^+)\)\cup (wN(x)\cap \Phi^+).$$  
This implies the following result.
\begin{lem}\label{inclusione}
For all $w, x\in W$, the following conditions are equivalent: 
\begin{enumerate}
\item $N(w)\subseteq N(wx)$; 
\item $wN(x)\subseteq \Phi^+$; 
 \item $N(wx)=N(w)\cup wN(x)$;
\item $\ell(wx)=\ell(w)+\ell(x)$.
\end{enumerate}
 Equivalently, for $w, v\in W$,  
$N(w)\subseteq N(v)$ if and only if $\ell(v)=\ell(w)+\ell(w^{-1}v)$.
\end{lem}

We also note that the right descents of $w\in W$ are the simple roots in $\ov N(w)$: 
\begin{equation}\label{discese}
D_r(w)=\ov N(w)\cap \Pi.
\end{equation}

\smallskip

\subsection{Ideals.}\label{ideals}
By \emph{the root poset of $\Phi$} (w.r.t. the basis $\Pi$) we intend  the partially ordered set whose underlying set is $\Phi^+$, with the order induced by  the root lattice.
The root poset could be equivalently defined as the transitive closure of the relation  $\al \lhd \be$ if and only if $\be - \al$ is a simple root.
The root poset hence is ranked by the height function and has the highest root $\theta$ as maximum.
\par

A {\it dual order ideal}, or {\it filter}, of $\Phi^+$ is a subset $I$ of $\Phi^+$ such that, if $\al\in I$ and $\be \geq \al$, then $\be \in I$.
\par

Let $\mf g$ be a complex simple Lie algebra, and $\mf h$ a Cartan subalgebra of $\mf g$ such that $\Phi$ is the root system of $\mf g$ with respect to $\mf h$. 
For all $\al\in \Phi$, we denote by $\mf g_\al$ the root space of $\al$. 
Moreover, we denote by $\mf b$ the Borel subalgebra corresponding to $\Phi^+$: $\mf b=\mf h + \mf n^+$, where $\mf n^+=\sum\limits_{\al\in \Phi^+}\mf g_\al=[\mf b, \mf b]$.
\par

Let $\mf i$ be an ideal of $\mf b$. Then $\mf i$ is $\mf h$-stable, hence there exists a subset $\Phi_\mf i$ of $\Phi^+$ such that 
$$
\mf i=\sum\limits_{\al\in \Phi_\mf i} \mf g_\al+ \mf  i\cap \mf  h.
$$
The condition $[\mf i, \mf b]\subseteq \mf i$ implies that $(\Phi_\mf i + \Phi^+) \cap \Phi \subseteq \Phi_\mf i$, i.e. that $\Phi_\mf i$ is a dual order ideal in the root poset. 
Conversely, if $I$ is a dual order ideal of $\Phi^+$, then 
$$
\mf i_I=\sum\limits_{\al\in I} \mf g_\al
$$ 
is an ideal of $\mf b$, included in $\mf n^+$. We call ad-nilpotent the ideals of $\mf b$ included in $\mf n^+$. 
Thus $\mf i\mapsto \Phi_\mf i$ is a bijection from the set of ad-nilpotent ideals of $\mf b$ to the set of dual order ideals of $\Phi^+$, with inverse map  $I\mapsto \mf i_I$. 
\par
 
It is easy to prove that an abelian ideal of  $\mf {b}$ is necessarily ad-nilpotent. 
The ideal $\mf i$ is abelian if and only if $\Phi_\mf i$ satisfies the abelian condition: $(\Phi_\mf i + \Phi_\mf i) \cap \Phi= \emptyset$.
\par

The dual order ideal $I$ of $\Phi^+$ is called {\it principal} if it has a minimum.
In such a case the ad-nilpotent ideal $\mf i_I$ is principal, being generated, as a $\mf b$-module, by any non-zero vector of the root space $\mf {g}_{\eta}$, where $\eta=\min I$. 
\par
 
We now briefly describe the affine root system associated to $\Phi$.
Following \cite{Kac}, we denote by $\de$ the basic imaginary root, 
and denote by $\wh \Phi$ the set of real roots of the untwisted affine Lie algebra associated with $\Phi$,
$$
\wh \Phi=\Phi+\ganz \de:=\{\al+k\de\mid \al\in
 \Phi, k\in
\ganz\}.
$$
Then 
$\gen\,\wh \Phi=\gen\,\Phi\oplus\real\de$, and $\wh \Phi$ is a crystallographic  affine root system in $\gen\,\wh \Phi$ endowed with the bilinear form obtained by extending the scalar product of $\gen\,\Phi$ to a positive semidefinite form with kernel $\real \de$.
If we take $\al_0=-\theta+\de$, then $\wh \Pi:=\{\al_0\}\cup \Pi$
is a root basis for $\wh \Phi$. The set of positive roots of $\wh
\Phi$  with respect to 
$\wh \Pi$ is 
\begin{equation}
\label{Phihat+}
\wh \Phi^+:=\Phi^+\cup ( \Phi+\ganz^+\de),
\end{equation}
where
$\ganz^+$ is the set of positive integers.
The affine Weyl group $\wh W$ associated to $W$ is the Weyl group of $\wh \Phi$. We extend the notation $N(w)$ to all $w\in \wh W$, so 
$N(w)=\{\al\in \wh \Phi^+\mid  w^{-1}(\al)\in -\wh \Phi^+\}$. 
\par 

The analogues of most results of Subsection \ref{Nw} hold in the affine case, too, possibly with slight restrictions. In particular, Formulas \eqref{supporto} and \eqref{N}, and Lemma \ref{inclusione} hold without changes. The equivalence of conditions (1), (2), and (3) extends to the affine case, except for $\Phi$ of type $A_1$. Thus, except for $\Phi$ of type $A_1$,  for all finite $N\subset \wh\Phi^+$, the following three conditions are equivalent:
\begin{enumerate}
\item $N$ and $\wh\Phi^+\setminus N$ are closed, 
\item there exists $w\in \wh W$ such that $N=N(w)$, 
\item $N$ and $\wh \Phi^+\setminus N$ are convex.
\end{enumerate}  
When  $\Phi$ is of type $A_1$, all subsets of  $\widehat \Phi^+$ are closed,  since the sum of two roots in $\widehat \Phi^+$ is never a root. Thus, in this case, the implications  (1) $\Rightarrow$ (2) and (1) $\Rightarrow$ (3) fail.
\par
We  recall  Peterson's bijection,  described in \cite{K}, between the abelian ideals and a special subset of the affine Weyl group. 
Let $\mf i$ be an abelian ideal of $\mf b$, and
$-\Phi_{\mf i}+\de:=\{-\al+\de\mid \al\in \Phi_{\mf i}\}.$
Then, there exists a (unique) element
$w_{\mf i}\in \wh W$ such that $N(w_{\mf i})=-\Phi_{\mf i}+\de$. 
Moreover, the map $\mf i\mapsto w_{\mf i}$ is a bijection from the
set of the abelian ideals of $\mf b$ to the set of all elements $w$ in 
$\wh W$ such that $N(w) \subseteq -\Phi^++\de$.
\par 

In order to simplify the notation, we identify the ad-nilpotent ideals with their set of roots: henceforward, we shall view such ideals as subsets of $\Phi^+$. 
\par 

We denote by $\mc I_\ab$ the set of abelian ideals in $\Phi^+$ and by 
$\wh W_\ab$ the corresponding set of elements in $\wh W$.
\par

For all $w\in \wh W_\ab$, $w^{-1}(-\theta + 2\de)$ is a long positive root and belongs to $\Phi$. 
For all long roots $\al\in \Phi^+$, let 
\begin{eqnarray}
\mc I_{\ab}(\al)=\{\mf i\in \mc I_\ab\mid w_{\mf i}^{-1}(-\theta + 2\de)= \al\}.
\end{eqnarray}

In the following statement we collect some results about the poset of the abelian ideals proved by Panyushev.
 
\begin{thm}\label{panyushev} (Panyushev \cite{Pany}).
\begin{enumerate}
 \item For any long root $\alpha$, $\mc I_{\ab}(\al)$ has a minimum and
a maximum. 
\item
The  maximal abelian ideals are exactly the maximal elements of the $\mc I_{\ab}(\al)$ with $\al\in \Pi_\ell$.  In particular, the maximal abelian ideals  are in bijection with the long simple roots. 
\item
For any pair of distinct  $\alpha, \alpha'$ in $\Pi_\ell$, any ideal in $\mc I_\ab(\al)$ is incomparable with any ideal in $\mc I_\ab(\al')$.
\end{enumerate}
\end{thm}
\par

The following characterization of $\mc I_\ab(\al)$ will be needed in the next sections. 

\begin{pro}\label{Ialfa}
Let $\mf i\in \mc I_\ab$. The following
are equivalent:
\begin{enumerate}
\item
\label{uno.}
there exists $\al\in \Pi_\ell$ such that  $\mf i\in \mc I_\ab(\al)$,
\item
\label{due.}
for all $\be,\ga \in \Phi^+$ such that $\be+\ga=\theta$, exactly one of $\be$ and $\ga$ belongs to $\mf i$.
\end{enumerate}
\end{pro}

\begin{proof}
If $\Phi=A_1$, the two assertions are trivially equivalent, so we may assume  $\Phi\neq A_1$.
We first observe that (\ref{due.}) is equivalent to the following statement: 
\begin{itemize}
\item[(3)]
for all  $\be',\ga' \in \wh\Phi^+$ such that $\be'+\ga' =-\theta+2\de$,
exactly one of $\be'$, $\ga'$ belongs to $N(w_\mf i)$.
\end{itemize}
Indeed, by Formula \eqref{Phihat+}, $\be', \ga' \in \wh\Phi^+$ are such that $\be'+\ga' =-\theta+2\de$,  if and only if $\be',\ga'\in -\Phi+\de$, thus $\be:=\de-\be'  \in \Phi^+$,
$\ga:=\de-\ga' \in \Phi^+$, and $\be+\ga=\theta$. Since, for all $\eta'\in \wh\Phi^+$, we have that $\eta'\in N(w_\mf i)$ if and only if $\de-\eta' \in \mf i$, assertion 
(3) is equivalent to (2).
\par

Now, we assume that $\mf i\in \mc I_\ab(\al)$ for some $\al\in \Pi_\ell$. 
Then, since $-\theta+2\de$ is a positive root, by Lemma \ref{inclusione}, $N(w_{\mf i}s_\al)=N(w_{\mf i})\cup\{-\theta+2\de\}$. 
This implies that  both $N(w_\mf i)\cup \{-\theta+2\de\}$ and its complement are closed.  
The closure of $\wh \Phi^+\setminus\big (N(w_\mf i)\cup \{-\theta+2\de\}\big)$, together with the the closure of  $N(w_\mf i)$,  implies (3). 
\par
It remains to prove that (3) implies that   $\mf i\in \mc I_\ab (\al)$ for some $\al\in \Pi_\ell$. 
Assume that (3) holds.  Since  $N(w_\mf i)\subseteq -\Phi^++\de$, we have that,  for all $\gamma\in N(w_\mf i)$, $-\theta + 2 \delta+\gamma$ is not a root: hence $N(w_\mf i)\cup \{-\theta+2\de\}$ is closed since
$N(w_\mf i)$ is.
Moreover, $\wh \Phi^+\setminus\big (N(w_\mf i)\cup \{-\theta+2\de\}\big)$ is closed by (3) and the fact that  $\wh \Phi^+\setminus N(w_\mf i)$ is closed.
Therefore, there exists $w\in \wh W$
such that $N(w_\mf i)\cup \{-\theta+2\de\}=N(w)$. By Lemma \ref{inclusione}, it follows that there exists $\al \in \widehat \Pi$ satisfying $w=w_\mf i s_\al$ and 
$ w_\mf i(\al)=-\theta+2\de$ . Indeed, $\alpha
\in \Pi_\ell$, since  it is a general fact that, for all $w\in \wh W_\ab$,
$w^{-1}(-\theta + 2\de)$ is a (long positive) root in $\Phi$.
\end{proof}

\smallskip

\subsection{Root polytopes.}\label{rootpolytopes}
For any subset $S$ of $\gen \,(\Phi)$, we denote by $\conv(S)$ the convex hull
of $S$ and by  $\conv_0(S)$ the convex hull of $S\cup\{\un 0\}$.
\par

We denote by $\mc P_{\Phi}$ the root polytope of $\Phi$, i.e. the convex hull of all roots in $\Phi$:
$$
\mc P_\Phi:=\conv(\Phi).
$$
We recall the results in \cite{C-M} that will be needed in the sequel.
Recall that $\theta=\sum_{i\in[n] }m_i \al_i$ is the highest root of $\Phi$, with respect to $\Pi$, and $\{\breve\omega_i\mid i\in [n]\}$ is the set of the fundamental co-weights of $\Phi$. 
We will think of $\breve\omega_i$ both as a functional and as a vector (the vector of $\gen\,\Phi$ defined by the condition  $(\al_j,\breve\omega_i)=\de_{ji}$, for all $j\in [n]$). 
\par

For every $I \subseteq [n]$, we let 
$$
V_I:=\{\al\in \Phi^+\mid (\al,\breve{\omega}_i)= m_i, \; \forall i \in I\}, \qquad F_I:=\conv(V_I)
$$ 
and we call $F_I$  the \emph{standard parabolic face associated with $I$}. For simplicity, we let $F_i=F_{\{i\}}$.
The standard parabolic faces $F_I$ are actually faces of the root polytope
$\mc P_{\Phi}$, since $\mc P_{\Phi}$ is included in the half-space $(-,\breve{\omega}_i)\leq m_i$. Moreover, the maximal root $\theta$ belongs to all standard parabolic faces.
\par

We denote by $\wh \Phi$ the affine root system associated with $\Phi$.
Let $\wh \Pi$ be an extension of $\Pi$ to a simple system of $\wh
\Phi$ and 
$\al_0\in \wh \Phi$ be such that $\wh \Pi=\Pi\cup
\{\al_0\}$.
Given $I \subseteq [n]$, we let $\Pi_I := \{\al_i \mid i \in I\}$,
$\Gamma_0(I)$ be the set of roots lying in the connected component of $\al_0$ in the Dynkin graph of $\wh \Pi\setminus \Pi_I$, and
\begin{itemize}
 \item[{}] $\ov I:=\{k\mid \al_k \not\in \Gamma_0(I)\}$,
\item[{}] $\partial I:=\{ j\mid \al_j\in \Pi_I, \text{ and }\exists\, \be \in 
\Gamma_0(I) \text{ such that } \be\not\perp\al_j\}.$
\end{itemize}
 
We recall that, for any subset $\Gamma$ of $\Phi$, we denote by $W\la\Gamma\ra$
the subgroup of $W$
generated by the reflections with respect to the roots in $\Gamma$, 
and  by $\Phi(\Gamma)$ the root subsystem generated by $\Gamma$, i.e.
$\Phi(\Gamma):=\{w(\ga)\mid w\in W\la\Gamma\ra,\ \ga\in \Gamma\}$.  
For each $\Gamma \subseteq \Pi$, 
let $\wh \Gamma=\Gamma \cup \{\al_0\}$ and $\wh{\Phi}(\wh
\Gamma)$ be the standard parabolic subsystem of $\wh \Phi$ generated by
$\wh \Gamma$.  
\par

In the following result we collect several properties of the standard
parabolic faces (see \cite{C-M} for proofs).
\begin{thm}
 \label{CM} 
Let $I \subset [n]$. Then: 
\begin{enumerate}
\item $V_I$ is a principal abelian ideal in $\Phi^+$,
\item $\{J \subset [n] \mid F_J=F_I\} = \{J \subset [n] \mid \partial I \leq
J \leq \ov{I} \}$, 
\item the dimension of $F_I$ is $n - |\ov I|$,
\item the stabilizer of $F_I$ in $W$ is $W\la \Pi \setminus \Pi_{\partial I}
\ra$,
\item $\{\ov I \mid I \subseteq [n] \} = \{ I \subseteq [n] \mid
\wh{\Phi}(\wh \Pi \setminus \Pi_I) \text{
is irreducible} \}$,
\item
the set $\{ \Pi\setminus \Pi_{\ov I} \} \cup \{ - \theta\}$ is a basis of the root subsystem $\Phi(V_I)$ generated by the roots in $F_I$.
\end{enumerate}
Moreover, the faces $F_{I}$, for $ I \in \mc F = 
\{ I \subseteq [n] \mid \wh{\Phi}(\wh \Pi \setminus \Pi_I) \text{
is irreducible} \}$,
form a complete set of representatives of the $W$-orbits.
In particular, the $f$-polynomial of $\mc P_\Phi$ is
$$
\sum\limits_{I \in \mc F} [W:W\la \Pi \setminus \Pi_{\partial I} \ra]
t^{n - | I|}.
$$ 
\end{thm}

The explicit description of the facets yields a description of $\mc P_\Phi$ as an intersection of a minimal set of half-spaces.

\begin{cor} 
Let 
$\Pi_\mc F=\left\{\al\in \Pi\mid \wh\Phi\left( \wh
\Pi\setminus
\{\al\} \right) 
\text{ is irreducible}\right\}$ and let $\mc L(W^\al)$ be a
set of representatives of the left cosets of $W$ modulo the subgroup  $W\la
\Pi\setminus\{\al\}\ra$.   Then 
$$
\mc P_\Phi=\big\{ x\mid (x, w\breve\omega_\al)\leq m_\al, \ 
\text{for all }\al\in \Pi_\mc F \text{ and } w\in \mc
L(W^\al )\big\}.
$$
Moreover, the above set is the minimal set of linear inequalities that 
defines $\mc P_\Phi$ as an intersection of half-spaces. 
\end{cor}
\smallskip

\subsection{Hyperplane arrangements.}\label{hyp-arrang} 
We follow \cite{StaHA} for notation and terminology concerning hyperplane arrangements. Given a hyperplane  arrangement $\mc H$ in $\mathbb R^n$, a region of $\mc H$ is a connected component of the complement $\mathbb R^n \setminus \bigcup_{H \in\mc  H} H$ of the hyperplanes. 
We let  $\mc L(\mc H)$ be the intersection poset of $\mc H$, i.e. the set of all nonempty intersection of hyperplanes in $\mc H$ (including $\mathbb R^n$ as the intersection over the empty set) partially ordered by reverse inclusion and $\chi_\mc H(t)$ be the characteristic polynomial of $\mc H$:  
$$
\chi_\mc H(t)=\sum_{x\in \mc L(\mc H)} \mu(x)t^{\dim x}
$$
($\mu(x)$ is the M\"{o}bius function from the bottom element $\mathbb R^n$ of $\mc L(\mc H)$).
A face of $\mc H$ is a set $\emptyset \neq F = \ov{R} \cap x$, where $\ov R$ is the closure of a region $R$ and $x \in \mc L(\mc H)$. 
We denote by $f_{\mc H}(x)$ the face polynomial of $\mc H$:
$f_{\mc H}(x) =\sum\limits_{F}  x^{\dim F}$ (the sum is over all faces $F$).
\par

The following result is due to Zaslavsky \cite{Z}. 

\begin{thm}
\label{zas}
The number of regions of an arrangement $\mc H$ in an $n$-dimensional real vector space is $(-1)^n \chi_\mc H(-1).$
\end{thm}

\section{A hyperplane arrangement associated to $\mc P_{\Phi}$}
\label{a hyp}
For all irreducible root systems $\Phi$, we consider a central hyperplane arrangement $\mc H_{\Phi}$ associated with the root polytope $\mc P_{\Phi}$ and hence with $\Phi$. 
The hyperplanes of $\mc H_{\Phi}$ are the hyperplanes through the origin spanned by the faces of $\mc P_{\Phi}$ of codimension 2:
$$
\mc H_{\Phi} := \{ \gen\,F \mid \textrm{$F$ face of $\mc
P_{\Phi}$ and $\codim F  = 2$}\}.
$$

For any face $F$ of $\mc
P_{\Phi}$, we denote by $V_F$ the roots lying on $F$:
$$
V_F=F\cap \Phi.
$$
The long roots in $V_F$ are the vertices of the face $F$, hence  $\gen\, F=\gen\, V_F$.  
Clearly, the following result holds.

\begin{pro}
\label{coni=}
Let $\Phi$ be an irreducible root system. Each  cone on a facet of
$\mc P_{\Phi}$ is the closure of a union of regions of 
the hyperplane arrangement $\mc H_{\Phi}$. 
\end{pro}

We are going to compute $\mc H_\Phi$ for all irreducible $\Phi$. It is clear that $\mc H_\Phi$ is $W$-stable. We will determine explicitly the dual vectors 
of the hyperplanes in  $\mc H_\Phi$ and their orbit structure. The results are listed in Table \ref{Tarrangement}. 
\par
By Theorem \ref{CM}, $\mc H_{\Phi}$ consists of the orbits of the hyperplanes spanned by the  standard parabolic faces of codimension 2 for the action of $W$. 
By Theorem \ref{CM}, (3) and (5), such faces are exactly the $F_I$ with
$I=\{i,j\}$ and $\wh\Phi(\wh \Pi\setminus\{\al_i, \al_j\})$ irreducible. 
\par
In Table \ref{Tn-2faces},  we list all the irreducible $(n-1)$-dimensional root
subsystems obtained by removing two finite nodes from the extended
Dynkin diagram. 
If $\al_i$ and $\al_j$ correspond to the removed nodes, which are crossed in the table, and we read the affine node as $-\theta$, then the resulting subsystem is the $(n-1)$-dimensional subsystem of $\Phi$ with root basis $\{\al_k \mid k \neq i,j \} \cup \{-\theta\}$. 
Thus, these are exactly the root subsystems $\Phi(V_F)$, i.e.  the
subsystems generated by the $V_F$,  for all the standard parabolic
$(n-2)$-faces $F$.  

\begin{table}[h]
\caption{The diagrams obtained by removing the crossed nodes yield the Dynkin diagrams of all $\Phi(V_F)$, $F$ standard parabolic $(n-2)$-face. }
\label{Tn-2faces}
\label{Tn-2faces}\centering
\begin{tabular}{l| l l}
\\
{$A_n$}
&
\parbox[c]{5cm}{\raisebox{5pt}{
\begin{tikzpicture}[scale=.6]
\An
\draw{(2.5,0)--(4.5,0)};
\draw[fill=white]{(3,0)circle(3pt)};
\draw[fill=white]{(4,0)circle(3pt)};
\draw{(2.8,.2)--(3.2,-.2) (2.8,-.2)--(3.2,.2)}; 
\draw{(3.8,.2)--(4.2,-.2) (3.8,-.2)--(4.2,.2)}; 
\node[below right] at (2.65,-.1){$\ss i$};
\node[below right] at (3.4,-.1){$\ss i+1\qquad (1\leq i\leq n-1)$};
\end{tikzpicture}
}}
& 
\\ \hline\\
$B_n$
&
\parbox[c]{5cm}{\raisebox{0pt}{
\begin{tikzpicture}[scale=.6]
\Bn
\draw{(.05,.8)--(.55,.6) (.05,.6)--(.55,.8)}; 
\node[anchor=north] at (.3,.7){$\ss 1$};
\draw{(7.8,.2)--(8.2,-.2) (7.8,-.2)--(8.2,.2)}; 
\node[anchor=north] at (8,0){$\ss n$};
\node[below] at (4,-1){$\phantom{n\geq 4}$};
\end{tikzpicture}
}}
&
\parbox[c]{5cm}{\raisebox{0pt}{
\begin{tikzpicture}[scale=.6]
\Bn
\draw{(6.8,.2)--(7.2,-.2) (6.8,-.2)--(7.2,.2)}; 
\node[below] at (7,0){$\ss{n-1}$};
\draw{(7.8,.2)--(8.2,-.2) (7.8,-.2)--(8.2,.2)}; 
\node[below] at (4,-1){$\phantom{n\geq 4}$};
\node[below] at (8.4,0){$\ss{n\phantom{-1}}$};
\node[below] at (4.3,-.7){${n\geq 4}$};
\end{tikzpicture}
}}
\\ \hline\\
{$C_n$}
&
\parbox[c]{5cm}{\raisebox{5pt}{
\begin{tikzpicture}[scale=.6]
\Cn
\draw{(7.8,.2)--(8.2,-.2) (7.8,-.2)--(8.2,.2)}; 
\node[below] at (8.4,0){$\ss{n\phantom{-1}}$};
\draw{(6.8,.2)--(7.2,-.2) (6.8,-.2)--(7.2,.2)}; 
\node[below] at (7,0){$\ss{n-1}$};
\end{tikzpicture}
}}
&
\\ \hline\\
{$D_n$}
&
\parbox[c]{5cm}{\raisebox{5pt}{
\begin{tikzpicture}[scale=.6]
\Dn
\draw{(.05,.8)--(.55,.6) (.05,.6)--(.55,.8)}; 
\node[anchor=north] at (.3,.7){$\ss 1$};
\draw{(7.45,.8)--(7.95,.6) (7.45,.6)--(7.95,.8)}; 
\node[anchor=north] at (8.1,.7){$\ss{r\phantom {-1}}$};
\phantom
{\draw{(7.45,-.8)--(7.95,-.6) (7.45,-.6)--(7.95,-.8)}; 
\node[below] at (8.1,-.7){$\ss{n\phantom {-1}}$};
}
\node[below] at (4,-.7){$\ss r=n-1 \text{ or \ } r=n$};
\end{tikzpicture}
}}
&
\parbox[c]{5cm}{\raisebox{5pt}{
\begin{tikzpicture}[scale=.6]
\Dn
\draw{(7.45,.8)--(7.95,.6) (7.45,.6)--(7.95,.8)}; 
\node[below] at (8.1,.7){$\ss {n-1}$};
\draw{(7.45,-.8)--(7.95,-.6) (7.45,-.6)--(7.95,-.8)}; 
\node[below] at (8.1,-.7){$\ss{n\phantom {-1}}$};
\end{tikzpicture}
}}
\\ \hline\\
{$E_6$}
&
\parbox[c]{5cm}{
\begin{tikzpicture}[scale=.6]
\Esei
\draw{(-.2, 1.2)--(.2,.8) (-.2,.8)--(.2,1.2)}; 
\node[below] at (0,.9){$\ss 1$};
\draw{(4-.2,1.2)--(4+.2,.8) (4-.2,.8)--(4+.2,1.2)}; 
\node[below] at (4,.9){$\ss 6$};
\end{tikzpicture}
}
&
\parbox[c]{5cm}{
\begin{tikzpicture}[scale=.6]
\Esei
\draw{(3-.2,1.2)--(3.2,.8) (3-.2,.8)--(3.2,1.2)}; 
\node[below] at (3,.9){$\ss i$};
\draw{(4-.2,1.2)--(4.2,.8) (4-.2,.8)--(4.2,1.2)}; 
\node[below] at (4,.9){$\ss j$};
\node[below right] at (2.5,-.1){$\ss\{i,j\}=\{1,2\}\text{ or \ }\{5,6\}$};
\end{tikzpicture}
}
\\
\\ \hline\\
{$E_7$}
&
\parbox[c]{5cm}{\raisebox{5pt}{
\begin{tikzpicture}[scale=.6]
\Esette
\draw{(3-.2,-1+.2)--(3+.2,-1-.2) (3-.2,-1-.2)--(3+.2,-1+.2)}; 
\node[anchor=north] at (3,-1-.1){$\ss 2$};
\draw{(6-.2,.2)--(6+.2,-.2) (6-.2,-.2)--(6+.2,.2)}; 
\node[anchor=north] at (6,-.1){$\ss 7$};
\end{tikzpicture}
}}
&
\parbox[c]{5cm}{\raisebox{5pt}{
\begin{tikzpicture}[scale=.6]
\Esette
\draw{(6-.2,.2)--(6+.2,-.2) (6-.2,-.2)--(6+.2,.2)}; 
\node[anchor=north] at (5,-.1){$\ss 6$};
\draw{(5-.2,.2)--(5+.2,-.2) (5-.2,-.2)--(5+.2,.2)}; 
\node[anchor=north] at (6,-.1){$\ss 7$};
\node[anchor=north] at (3,-1-.1){$\phantom{\ss 2}$};
\end{tikzpicture}
}}
\\ \hline\\
{$E_8$}
&
\parbox[c]{5cm}{\raisebox{5pt}{
\begin{tikzpicture}[scale=.6]
\Eotto
\draw{(0-.2,.2)--(0+.2,-.2) (0-.2,-.2)--(0+.2,.2)}; 
\node[anchor=north] at (0,-.1){$\ss 1$};
\draw{(1-.2,.2)--(1+.2,-.2) (1-.2,-.2)--(1+.2,.2)}; 
\node[anchor=north] at (1,-.1){$\ss 3$};
\phantom{\draw{(2-.2,-1+.2)--(2+.2,-1-.2) (2-.2,-1-.2)--(2+.2,-1+.2)}; 
\node[anchor=north] at (2,-1-.1){$\ss 2$};}
\end{tikzpicture}
}}
&
\parbox[c]{5cm}{\raisebox{5pt}{
\begin{tikzpicture}[scale=.6]
\Eotto
\draw{(0-.2,.2)--(0+.2,-.2) (0-.2,-.2)--(0+.2,.2)}; 
\node[anchor=north] at (0,-.1){$\ss 1$};
\draw{(2-.2,-1+.2)--(2+.2,-1-.2) (2-.2,-1-.2)--(2+.2,-1+.2)}; 
\node[anchor=north] at (2,-1-.1){$\ss 2$};
\end{tikzpicture}
}}
\\ \hline\\
{$F_4$}
&
\parbox[c]{5cm}{\raisebox{5pt}{
\begin{tikzpicture}[scale=.6]
\Fquattro
\draw{(4-.2,.2)--(4+.2,-.2) (4-.2,-.2)--(4+.2,.2)}; 
\node[anchor=north] at (4,-.1){$\ss 4$};
\draw{(3-.2,.2)--(3+.2,-.2) (3-.2,-.2)--(3+.2,.2)}; 
\node[anchor=north] at (3,-.1){$\ss 3$};
\end{tikzpicture}
}}
\\ \hline\\
{$G_{2}$}
&
\parbox[c]{5cm}{
\begin{tikzpicture}[scale=.6]
\Gdue
\draw{(.6,.15)--(.45,0)--(.6,-.15)};
\draw{(0-.2,.2)--(0+.2,-.2) (0-.2,-.2)--(0+.2,.2)}; 
\node[anchor=north] at (0,-.1){$\ss 1$};
\draw{(1-.2,.2)--(1+.2,-.2) (1-.2,-.2)--(1+.2,.2)}; 
\node[anchor=north] at (1,-.1){$\ss 2$};
\end{tikzpicture}
}
\\
\end{tabular}
\end{table}
\bigskip

\begin{pro}
\label{iperpiani}
Let $\Phi$ be an irreducible root system. There exists a subset $H_{\Phi} \subseteq [n]$ such that 
$$
\mc H_{\Phi} = \{ w (\breve{\omega}_{k}) ^{\perp}  \mid w \in W, k \in H_{\Phi}\}.
$$ 
Furthermore, $H_{\Phi}$ is contained in the set of $h\in [n]$ such that the standard parabolic subsystem of $\Phi$ generated by $\Pi\setminus \{\al_h\}$ is irreducible. 
\end{pro}

\begin{proof}
Let $F$ be a $(n-2)$-face of $\mc P$. By Theorem \ref{CM}, the 
roots in $F$ generate an irreducible subsystem $\Phi(V_F)$ of rank $n-1$. Hence also the parabolic closure  $\ov{\Phi(V_F)} = \gen\, F \cap \Phi$ is irreducible
and has rank $n-1$. By \cite[Chapitre 6, \S 1, Proposition 24]{Bou}, it follows that there exist  $k\in [n]$ and  $w$ in $W$ such that the standard parabolic subsystem generated by $\Pi\setminus \{\al_k\}$ is irreducible and is transformed by $w$ into $\ov{\Phi(V_F)}$. Hence $\gen\, F=w(\breve\omega_k)^\perp$.
\end{proof}

For $k\in [n]$, let $[\breve\omega_k^\perp]$ be the orbit of $\breve\omega_k^\perp$ under the action of $W$,
$$
[\breve \omega_k^\perp]:=\{w(\breve\omega_k)^\perp\mid w\in W\}.
$$  
We note that, for all $k, j\in [n]$,  $[\breve \omega_k^\perp] = [\breve
\omega_j^\perp]$ if and only if the standard parabolic subsystems $\Phi(\Pi
\setminus \{\al_k\})$ and $\Phi(\Pi \setminus \{\al_j\})$ can be transformed
each into the other by $W$. In fact, $\Phi(\Pi \setminus \{\al_k\}) = \breve
\omega_k ^{\perp} \cap \Phi$ and hence for each $w\in W$, $w( \breve \omega_k
^{\perp})= \breve \omega_j ^{\perp}$ if and only if $w(\Phi(\Pi \setminus
\{\al_k\}))= \Phi(\Pi \setminus \{\al_j\})$. 
\par
Thus, by Proposition \ref{iperpiani}, in order to determine $\mc H_\Phi$ and its orbit structure for the action of $W$, we need to determine, for each standard parabolic $(n-2)$-face $F$,  the parabolic
closure $\ov{\Phi(V_F)}$ and  a standard parabolic subsystem in its orbit under
the action of $W$. 
Moreover,  we must check whether these standard parabolic subsystems can be transformed each into the other by $W$.
\par

For computing explicitly the parabolic closures of the subsystems in Table
\ref{Tn-2faces},  we use the following lemma.

\begin{lem}
\label{chiusuraparabolica} 
Let $I\subseteq[n]$ be such that 
$\dim F_I = n-|I|$. For any $b\in\real^+$, let $S_b=\{\ga\in \Phi\mid
c_i(\ga)=\frac{m_i}{b} \ \forall \ i\in I\}$, and let $d$ be the maximal
positive
number such that $S_d\neq \emptyset$.
 
Then:
\item{(1)} $S_d$ has a minimum (in the root poset) $\al$ and $\{\al\}\cup
(\Pi\setminus \Pi_I)$ is a root
basis for $\ov{\Phi(V_I)}$;
\item{(2)}  $\Phi(V_I) = \ov{\Phi(V_I)}$ if and only if $d=1$.
In particular, if $\gcd\{m_i\mid i\in I\}=1$, then $\Phi(V_I)=\ov{\Phi(V_I)}$ . 
\end{lem}

\begin{proof}
Let $b\in\real^+$ be such that $S_b \neq \emptyset$.
First notice that $\frac{m_i}{b}$ must be an
integer between $1$ and $m_i$, for all $i\in I$, and, in particular, $b$ must be
a rational number between $1$ and $\min\{m_i\mid i\in I\}$. Moreover, if $b=z/t$ with
$z$
and $t$ relatively prime integers, then $z|m_i$ for all $i\in I$: therefore, if
$\gcd\{m_i\mid i\in I\}=1$, then $d=1$.
Recall from Theorem \ref{CM} that $\Phi(V_I)$ has basis $(\Pi \setminus \Pi_I)
\cup \{ -\theta \}$, hence 
$\Phi(V_I)= S_1^\pm\cup \Phi(\Pi \setminus \Pi_I)$, where we set
$S_b^\pm=S_b\cup-S_b$, for all $b$. Moreover,
since
$\gen\,V_I$ is generated by $\theta$ and $\Pi
\setminus \Pi_I$,
$\ov{\Phi(V_I)}=\left( \bigcup_{b\in \real^+} S_b^\pm \right)\cup \Phi(\Pi
\setminus \Pi_I)$. 
Therefore,  we obtain that $\Phi(V_I)=\ov{\Phi(V_I)}$ if and only if $d=1$. It
remains to prove
(1). 
Let $B$ be the basis of $\ov{\Phi(V_I)}$ such that the set of positive roots of  
 $\ov{\Phi(V_I)}$ with respect to $B$ is $\ov{\Phi(V_I)}\cap \Phi^+$. Then 
it is clear that $B$ must contain $\Pi \setminus \Pi_I$ and hence any
minimal root in $S_d$.
It follows that $S_d$ has a minimum $\al$ and that $B=\{\al\}\cup
(\Pi \setminus \Pi_I)$. 
\end{proof} 

In the proof of the following proposition, we detail the explicit construction
of the parabolic closures in the few cases in which $\Phi(V_I)$ is not
parabolic.   

\begin{pro}\label{BEF}
All the root subsystems that occur in Table \ref{Tn-2faces} are parabolic
subsystems of $\Phi$, except the second one of the $B_n$ case, the first one of
$E_8$, and the one of $F_4$. In these cases $\ov{\Phi(V_I)}$ is of type $B_{n-1}$,
$E_7$, and $B_3$, respectively. 
\end{pro} 

\begin{proof}
Apart from three exceptions, the diagrams in Table \ref{Tn-2faces} are obtained by removing nodes of indexes $i, j$ such that $\gcd(m_i, m_j)=1$, thus they correspond to parabolic subsystems, by Lemma \ref{chiusuraparabolica}. 
The three exceptions are obtained when removing nodes of indexes: $i=n-1, j=n$ in type $B_n$; $i=1,j=3$ in type $E_8$; $i=3, j=4$ in type $F_4$, where in all these cases we have $\gcd(m_i,m_j)=2$. In these cases, the
subsystems $\Phi(V_I)$ are nonparabolic subsystems of type $D_{n-1}$, $A_7$, and $A_3$, respectively. The second one is nonparabolic by Lemma  \ref{chiusuraparabolica}. The first and the third are non parabolic since in types $B_n$ and $F_4$ there are no parabolic subsystems of types $D_{n-1}$ and $A_3$, respectively. 
By Lemma \ref{chiusuraparabolica}, in these cases a basis for $\ov{\Phi(V_I)}$ can be obtained from  $\Pi$ by removing $\al_i$ and $\al_j$ and adding the
minimal root $\al$ such that $c_i(\al)=\frac{m_i}{2}$, 
$c_j(\al)=\frac{m_j}{2}$. By inspection, we see that this minimal root is $\al=\al_{n-1}+\al_n$ in case $B_n$, $\al=\al_1+\al_2+2\al_3+2\al_4+\al_5$ in case $E_8$,
and  $\al=\al_{2}+2\al_3+\al_4$ in case $F_4$. By a direct check, we see that adding this root to $\Pi\setminus\{\al_i, \al_j\}$ produces diagrams of type $B_{n-1}$, $E_7$, and $B_3$, respectively. 
\end{proof}

It is clear that  if $\Phi(\Pi \setminus \{\al_k\})$ and $\Phi(\Pi \setminus
\{\al_j\})$ have different Dynkin diagrams, then they cannot be transformed each into the other by $W$, hence $[\breve \omega_k^\perp] \neq
[\breve \omega_j^\perp]$.
\par

However, in general the Dynkin graph does not determine the orbit of the
parabolic subsystem; two standard parabolic subsystems having isomorphic Dynkin
diagrams may  or may not be transformed one into the other by the Weyl group.
\par 

From Table \ref{Tn-2faces} and Proposition \ref{BEF} we see that 
in all cases except $A_n$ and $D_n$,  the Dynkin diagram of $\ov{\Phi(V_F)}$
determines uniquely the orbit of its hyperplane, since in these cases there is a unique standard parabolic root subsystem with the same  Dynkin diagram as $\ov{\Phi(V_F)}$, for all $(n-2)$-standard parabolic faces $F$.
The cases $A_n$ and $D_n$  require a (brief) direct 
check. For type $A_n$, it suffices to notice that there exist exactly two
standard parabolic subsystems of rank $n-1$, which are of type $A_{n-1}$ and
can be transformed each into the other by $W$. In type $D_n$, the two standard
parabolic faces obtained for $I=\{1,n\}$ and  $I=\{1,n-1\}$ give two 
parabolic subsystems of type $A_{n-1}$: these can be transformed each into the
other if and only if $n$ is odd.
\par
In Table \ref{Tarrangement}, for each $\Phi$ we list the standard parabolic subsystems corresponding to the $(n-2)$-faces and specify the orbit $[\omega_k^\perp]$ of the hyperplane they span. Thus, $\mc H_\Phi$ is the union of the $[\omega_k^\perp]$ that occur in the table. 
The listed orbits  are distinct,  unless explicitly noted.
\par 

We remark that though different standard  parabolic faces $F$ and $F'$ belong to different orbits under the action of $W$ (Theorem \ref{CM}), it may happen that the hyperplanes $\gen\, F$ and $\gen\, F'$ belong to the same orbit of $W$. This happens in type $A_n$ (for all $n\geq 3$) and $D_n$ with $n$ odd, as we can see by comparing Table \ref{Tn-2faces}, where the listed subsystems are in bijection with the orbits of $(n-2)$-faces,  and Table \ref{Tarrangement}, where  the orbits of hyperplanes generated by the $(n-2)$-faces are listed.

\newcommand\Anfdue
{\draw[gray!50]{(1,0)--(2,0)};
\draw{(2,0)--(2.5,0) (6.5,0)--(8,0)};
\draw[dashed]{ (2.5,0)--(6.5,0)};
\foreach \x in {2, 7, 8}
\draw[fill=white]{(\x,0) circle(3pt)};
\draw[gray!50][fill=white]{(1,0) circle(3pt)};
}

\newcommand\Bnfduea
{\draw{(1,0)--(2,0)--(2.5,0) 
 (5.5,0)--(6,0)--(7,0)};
\draw[dashed]{ (2.5,0)--(5.5,0)};
\draw[gray!50]{(7,-1.5pt)--(8,-1.5pt)};
\draw[gray!50]{(7,1.5pt)--(8,1.5pt)};
\draw[gray!50]{(7.4,.15)--(7.55,0)--(7.4,-.15)};
\foreach \x in {1,2,6,7}
\draw[fill=white]{(\x,0) circle(3pt)};
\draw[gray!50][fill=white]{(8,0) circle(3pt)};}

\newcommand\Bnfdueb
{\draw[gray!50]{(1,0)--(2,0)}; 
\draw{(2,0)--(2.5,0) 
 (5.5,0)--(6,0)--(7,0)};
\draw[dashed]{ (2.5,0)--(5.5,0)};
\draw{(7,-1.5pt)--(8,-1.5pt)};
\draw{(7,1.5pt)--(8,1.5pt)};
\draw{(7.4,.15)--(7.55,0)--(7.4,-.15)};
\foreach \x in {2,6,7,8}
\draw[fill=white]{(\x,0) circle(3pt)};
\draw[gray!50][fill=white]{(1,0) circle(3pt)};
}

\newcommand\Cnfdue
{\draw[gray!50]{(1,0)--(2,0)}; 
\draw{(2,0)--(2.5,0) (5.5,0)--(6,0)--(7,0)};
\draw[dashed]{ (2.5,0)--(5.5,0)};
\draw{(7,-1.5pt)--(8,-1.5pt)};
\draw{(7,1.5pt)--(8,1.5pt)};
\draw{(7.6,.15)--(7.45,0)--(7.6,-.15)};
\foreach \x in {2,6,7,8}
\draw[fill=white]{(\x,0) circle(3pt)};
\draw[gray!50][fill=white]{(1,0) circle(3pt)};}

\newcommand\Dnfduea
{\draw{(1,0)--(2.5,0)  (5.5,0)--(7,0)--(7.7,-.7)};
\draw[dashed]{(2.5,0)--(5.5,0)};
\draw[gray!50]{(7,0)--(7.7,.7)}; 
\foreach \x in {1,2,6,7}
\draw[fill=white]{(\x,0) circle(3pt)};
\draw[gray!50][fill=white]{(7.7,.7) circle(3pt)};
\draw[fill=white]{(7.7,-.7) circle(3pt)};}

\newcommand\Dnfdueb
{\draw[gray!50]{(1,0)--(2,0)}; 
\draw{(2,0)--(2.5,0) (5.5,0)--(7,0)--(7.7,.7) (7,0)--(7.7,-.7)};
\draw[dashed]{ (2.5,0)--(5.5,0)};
\foreach \x in {2,6,7}
\draw[fill=white]{(\x,0) circle(3pt)};
\draw[gray!50][fill=white]{(1,0) circle(3pt)};
\draw[fill=white]{(7.7,.7) circle(3pt)};
\draw[fill=white]{(7.7,-.7) circle(3pt)};}

\newcommand\Eseifduea
{\foreach \x in {0, 1, 2}
\draw{(\x,0)--(\x+1,0)};
\draw[gray!50]{(3,0)--(4,0)};
\draw{(2,0)--(2,-1)};
\foreach \x in {0, 1, 2, 3}
\draw[fill=white]{(\x,0) circle(3pt)};
\draw[gray!50][fill=white]{(4,0) circle(3pt)};
\draw[fill=white]{(2,-1) circle(3pt)};}

\newcommand\Eseifdueb
{\foreach \x in {0, 1, 2,3}
\draw{(\x,0)--(\x+1,0)};
\draw[gray!50]{(2,0)--(2,-1)};
\foreach \x in {0, 1, 2,3,4}
\draw[fill=white]{(\x,0) circle(3pt)};
\draw[gray!50][fill=white]{(2,-1) circle(3pt)};}

\newcommand\Esettefduea
{\foreach \x in {1, 2,3,4,5}
\draw{(\x,0)--(\x+1,0)};
\draw[gray!50]{(3,0)--(3,-1)};
\foreach \x in {1, 2,3,4,5,6}
\draw[fill=white]{(\x,0) circle(3pt)};
\draw[gray!50][fill=white]{(3,-1) circle(3pt)};}

\newcommand\Esettefdueb
{\foreach \x in {2,3,4, 5}
\draw{(\x,0)--(\x+1,0)};
\draw[gray!50]{(1,0)--(2,0)};
\draw{(3,0)--(3,-1)};
\foreach \x in {2,3,4,5,6}
\draw[fill=white]{(\x,0) circle(3pt)};
\draw[gray!50][fill=white]{(1,0) circle(3pt)};
\draw[fill=white]{(3,-1) circle(3pt)};}

\newcommand\Eottofduea
{\foreach \x in {0,1,2,3,4,5}
\draw{(\x,0)--(\x+1,0)};
\draw[gray!50]{(2,0)--(2,-1)};
\foreach \x in {0,1,2,3,4,5,6}
\draw[fill=white]{(\x,0) circle(3pt)};
\draw[gray!50][fill=white]{(2,-1) circle(3pt)};}

\newcommand\Eottofdueb
{\foreach \x in {0,1,2,3,4}
\draw{(\x,0)--(\x+1,0)};
\draw[gray!50]{(5,0)--(6,0)};
\draw{(2,0)--(2,-1)};
\foreach \x in {0,1,2,3,4,5,6}
\draw[fill=white]{(\x,0) circle(3pt)};
\draw[fill=white]{(2,-1) circle(3pt)};
\draw[gray!50][fill=white]{(6,0) circle(3pt)};}

\newcommand\Fquattrofdue
{\draw{(1,0)--(2,0)};
\draw[gray!50]{(3,0)--(4,0)};
\draw{(2,1.5pt)--(3,1.5pt)};
\draw{(2,-1.5pt)--(3,-1.5pt)};
\draw{(2.4,.15)--(2.55,0)--(2.4,-.15)};
\foreach \x in {1,2,3}
\draw[fill=white]{(\x,0) circle(3pt)};
\draw[gray!50][fill=white]{(4,0) circle(3pt)};}

\newcommand\Gduefdue
{\foreach\y in{0,2,-2}
\draw[gray!50]{(0,\y pt)--(1,\y pt)};
\draw[gray!50]{(.6,.15)--(.45,0)--(.6,-.15)};
\draw[fill=white]{(1,0) circle(3pt)};
\draw[gray!50][fill=white]{(0,0) circle(3pt)};}

\begin{table}[ht]
\caption{Parabolic subsystems corresponding to $\mc{H}_{\Phi}$ and their orbits. 
The black parts of the diagrams are the Dynkin diagrams of the $\ov{\Phi(V_F)}$. The gray nodes correspond to some $\al_i$ such that  $\Phi(\Pi\setminus \{\al_i\})$ belongs to the $W$-orbit of $\ov{\Phi(V_F)}$, so that $\gen\,\ov{\Phi(V_F)}$ belongs to the orbit of $[\omega_i^\perp]$.}
\label{Tarrangement}\centering

\begin{tabular}{l| l l}
\\
{$A_n$}
&
\parbox[c]{5cm}{
\begin{tikzpicture}[scale=.6]
\Anfdue
\draw (0.7, .3)--(1.3, .3)--(1.3, -.3)--(.7,-.3)-- cycle;
\node[below] at (4.5,-.2) {$[\omega_1^\perp]=[\omega_n^\perp]$}; 
\end{tikzpicture}
}
&
\\ \hline\\
$B_n$
&
\parbox[c]{5cm}{
\begin{tikzpicture}[scale=.6]
\Bnfduea
\draw (7.7, .3)--(8.3, .3)--(8.3, -.3)--(7.7,-.3)-- cycle;
\node[below] at (4.5,-.2) {$[\omega_n^\perp]$};
\end{tikzpicture}
}
&
\parbox[c]{5cm}{
\begin{tikzpicture}[scale=.6]
\Bnfdueb
\draw (0.7, .3)--(1.3, .3)--(1.3, -.3)--(.7,-.3)-- cycle;
\node[below] at (4.5,-.2) {$[\omega_1^\perp]$\quad ${n\geq 4}$}; 
\end{tikzpicture}
}
\\ \hline\\
{$C_n$}
&
\parbox[c]{5cm}{
\begin{tikzpicture}[scale=.6]
\Cnfdue
\draw (0.7, .3)--(1.3, .3)--(1.3, -.3)--(.7,-.3)-- cycle;
\node[below] at (4.5,-.2) {$[\omega_1^\perp]$}; 
\end{tikzpicture}
}
& 
\\ \hline\\
{$D_n$}
&
\parbox[c]{5cm}{
\begin{tikzpicture}[scale=.6]
\Dnfduea
\draw (7.4, 1)--(8, 1)--(8, .4)--(7.4,.4)-- cycle;
\node[below right] at (2,-.2) {$[\omega_{n-1}^\perp]$, $[\omega_n^\perp]$}; 
\node[below right] at (2,-1.4) {
$[\omega_n^\perp]=[\omega_{n-1}^\perp]\text{ if and only if $n$ is odd}$};
\end{tikzpicture}
}
&
\parbox[c]{5cm}{
\begin{tikzpicture}[scale=.6]
\Dnfdueb
\draw (0.7, .3)--(1.3, .3)--(1.3, -.3)--(.7,-.3)-- cycle;
\node[below] at (4.5,-.2) {$[\omega_1^\perp]$}; 
\node[below] at (4.5,-1.4) {\phantom{$[\omega_1^\perp]$}}; 
\end{tikzpicture}
}
\\ \hline\\
{$E_6$}
&
\parbox[c]{5cm}{
\begin{tikzpicture}[scale=.6]
\Eseifduea
\draw (3.7, 0.3)--(4.3, 0.3)--(4.3, -0.3)--(3.7, -0.3)-- cycle;
\node[below] at (2,-1.2){ $[\omega_1^\perp]=[\omega_6^\perp]$};
\end{tikzpicture}
}
&
\parbox[c]{5cm}{
\begin{tikzpicture}[scale=.6]
\Eseifdueb
\draw (1.7, -.7)--(2.3, -.7)--(2.3, -1.3)--(1.7,-1.3)-- cycle;
\node[below] at (2,-1.2){ $[\omega_2^\perp]$};
\end{tikzpicture}
}
\\
\hline\\
{$E_7$}
&
\parbox[c]{5cm}{
\begin{tikzpicture}[scale=.6]
\Esettefduea
\draw (2.7, -.7)--(3.3, -.7)--(3.3, -1.3)--(2.7,-1.3)-- cycle;
\node[below] at (3.5,-1.2){ $[\omega_2^\perp]$};
\end{tikzpicture}
}
&
\parbox[c]{5cm}{
\begin{tikzpicture}[scale=.6]
\Esettefdueb
\draw (0.7, .3)--(1.3, .3)--(1.3, -.3)--(.7,-.3)-- cycle;
\node[below] at (3.5,-1.2){ $[\omega_1^\perp]$};
\end{tikzpicture}
} 
\\ \hline\\
{$E_8$}
&
\parbox[c]{5cm}{
\begin{tikzpicture}[scale=.6]
\Eottofduea
\draw (1.7, -.7)--(2.3, -.7)--(2.3, -1.3)--(1.7,-1.3)-- cycle;
\node[below] at (3,-1.2){ $[\omega_2^\perp]$};
\end{tikzpicture}
}
&
\parbox[c]{5cm}{
\begin{tikzpicture}[scale=.6]
\Eottofdueb
\draw (5.7, .3)--(6.3, .3)--(6.3, -.3)--(5.7,-.3)-- cycle;
\node[below] at (3,-1.2){ $[\omega_8^\perp]$};
\end{tikzpicture}
} 
\\ \hline\\
{$F_4$}
&
\parbox[c]{5cm}{
\begin{tikzpicture}[scale=.6]
\Fquattrofdue
\draw (3.7, .3)--(4.3, .3)--(4.3, -.3)--(3.7,-.3)-- cycle;
\node[below] at (2.5,-.2){ $[\omega_4^\perp]$};
\end{tikzpicture}
}
\\ \hline\\
{$G_{2}$}
&
\parbox[c]{5cm}{
\begin{tikzpicture}[scale=.6]
\Gduefdue
\draw (-.3, .3)--(.3, .3)--(.3, -.3)--(-.3,-.3)-- cycle;
\node[below] at (.5,-.2){ $[\omega_1^\perp]$};
\end{tikzpicture}
}
\\
\end{tabular}
\end{table}

\section{The facets of  $\mc P_{\Phi}$ and the regions of 
$\mc H_{\Phi}$}
\label{ac-si}
As already noted in Proposition \ref{coni=}, for all irreducible root systems
$\Phi$, the cones on the maximal faces of $\mc P_{\Phi}$ are unions of
regions of the hyperplane arrangement $\mc H_{\Phi}$. In this section, we show that, 
for types $A$ and
$C$, actually the cones on the maximal
faces of
$\mc P_{\Phi}$ are precisely the regions of the hyperplane
arrangement
$\mc H_{\Phi}$. On the contrary, for types $B$ and $D$, this is not the
case and
there are hyperplanes of $\mc H_{\Phi}$ intersecting the interior of some
facets of $\mc P_{\Phi}$. 
\par
Moreover, we recall the Young diagram formalism for representing the positive root poset of types $A$ and $C$ and give many examples.  

\subsection{Type $A$} 
Let $\Phi$ be of type $A_n$ and omit
the
subscript $\Phi$ everywhere. Recall that, in our convention, 
the positive roots are all the sums $\sum\limits_{k\in [i,j]}\al_k$,  with
$1\leq i\leq j\leq n$, where $\al_1, \dots, \al_n$ are the simple roots.
To simplify notation, we set $\al_{i,j}:= \sum_{h=i}^{j}
\al_h$, for all $1 \leq i \leq j \leq n$.

As noted in Table \ref{Tarrangement}, 
the arrangement $\mc H$ is given by the orbit $[\breve{\omega}_1^{\perp}]$
of the hyperplane orthogonal to the first fundamental coweight 
$\breve{\omega}_1$ (which coincides with the orbit $[\breve{\omega}_n^{\perp}]$ of the hyperplane orthogonal to the 
$n$-th fundamental coweight $\breve{\omega}_n$).
Since the stabilizer of
$\breve{\omega}_{1}$ in the Weyl group $W$ (contragredient representation)
is the parabolic subgroup 
$W\la \Pi \setminus \{ \al_1 \} \ra$,
the orbit of $\breve{\omega}_1$ is
obtained 
by acting with the set of the minimal coset representatives
$W^{1}:= \{s_i \cdots s_1 \mid i = 0,
\ldots n\}$ (for simplicity, we write $s_k$ instead of $s_{\al_k}$, for $k\in
[n]$). For
$i \in [1, n+1]$, let  
$\breve{\omega}_{1,i}:= s_{i-1}\cdots s_1 \left(\breve{\omega}_{1}\right)$. 
We
have that 
$$
\breve{\omega}_{1,i} = \left\{ \begin{array}{ll} 
\breve{\omega}_{1}, & \textrm{if $i = 1$,} \\
\breve{\omega}_{i} - \breve{\omega}_{i-1}, & \textrm{if $i \notin \{ 1,n+1\}$,}
\\                  
- \breve{\omega}_{n}, & \textrm{if $i = n+1$.} 
                                  \end{array} \right.
$$

The hyperplanes of $\mc H$ are exactly the hyperplanes
generated by the
sets of 
roots $(\Pi\cup \{-\theta\})\setminus \{\al_i, \al_{i+1}\}$, 
for $i=0, \dots, n$, where we intend 
$\al_0=\al_{n+1}=-\theta$ (by Theorem \ref{CM}, we already knew that 
$\mc H$ should contain such hyperplanes, for $i=1, \ldots, n-1$). 
Note that, in the usual coordinate presentation of the root system of type $A_n$ in the vector space 
$V= \{\sum \varepsilon_i=0\} \subseteq \mathbb R^{n+1}$ (see, for example, \cite{Bou}), the hyperplanes of 
$\mc H$ are the intersections of the coordinate hyperplanes of $\mathbb R^{n+1}$ with $V$. Hence, since $W$ acts on $V$ as the group of permutations of $\{\varepsilon_1, \dots, \varepsilon_{n+1}\}$,  $W$ acts on $\mc H$ as its group of permutations. 
\par
The one-dimensional subspaces of $\mc L(\mc H)$ are exactly the spaces
$\real\al$,
for all $\al \in \Phi$. 
In fact,
let $\mc{I}_k=\bigcap_{i=1}^k\breve{\omega}_{1, i}^\perp$ and notice that, for
$k=1,\dots, n$, 
$\mc I_k$  contains 
the simple roots $\al_j$ for $j>k$ and does not contain the simple roots
$\al_j$ for 
$j\leq k$; in particular $\mc I_1 \supsetneq \cdots \supsetneq \mc
I_n$. 
This implies 
that $\breve{\omega}_{1, 1}, \dots, \breve{\omega}_{1, n}$ are linearly
independent, and 
that $\mc I_k=\gen \,(\al_j\mid j>k)$. 
In particular $\mc I_{n-1}=\real \al_n$ whence, since $\mc H$ is
$W$-stable, $\real\al$ 
belongs to $\mc L(\mc H)$ for all $\al \in \Phi$ and, since $W$ is
$(n-1)$--fold 
transitive on $\mc H$, these are exactly the one-dimensional subspaces of
$\mc L(\mc H)$.

Consider  the two open halfspaces determined by $\breve{\omega}_{1, i}^\perp$:
\begin{itemize}
 \item[] $\breve{\omega}_{1, i}^+=\{x\in V\mid (x, \breve{\omega}_{1, i})> 
0\}
$,  
\item[] $\breve{\omega}_{1, i}^-=\{x\in V\mid ( x, \breve{\omega}_{1, i} )< 
0\}
.$
\end{itemize}
and the regions of $\mc H$ 
 
$$
\bigcap_{j=1}^{n+1}\breve{\omega}_{1, j}^{\sigma_j},
$$
where $\sigma_j\in \{+,-\}$. The regions are not empty, except exactly those two
with 
the $\sigma_j$ either all equal to $+$, or all equal to $-$. In fact, these two
sets are clearly empty.
On the other hand, by Proposition \ref{coni=}, the number of regions of 
$\mc H$ is greater than or equal to 
the number of facets of the root polytopes which, 
by Theorem \ref{CM}, is equal to 
$$
\sum_{i=1}^{n} [W: W\la \Pi \setminus \al_i \ra] = \sum _{i=1}^{n}
\frac{(n+1)!}{(i)!\,  (n+1-i)!} = 2^{n+1}-2.
$$
Hence the number of regions of 
$\mc H$ and the number of facets of $\mc P$ are equal to $2^{n+1}-2$.
Thus we have proved the following theorem.
\begin{thm}
\label{coincidono}
 If the root system $\Phi$ is of type $A$, the closures of the regions of the
hyperplane
arrangement $\mc H_{\Phi}$ coincide with the cones on the facets of the
root polytope $\mc P_{\Phi}$.
\end{thm}

If we set, for
all $i\in
[n]$,
$M_i=\left\{\al\in \Phi^+\mid \left( \al, \breve\omega_{i} 
\right)>0\right\}$ 
(see Section \ref{idealiabeliani} for the role of the sets $M_i$, $i \in [n]$,
in the theory of abelian ideals)
and 
$$
R_i=\bigcap_{j=1}^{n+1}\breve{\omega}_{1, j}^{\sigma_j}
$$
with $\sigma_j=+$ for  $1\leq j\leq i$  and  $\sigma_j=-$  otherwise,
then
$$
M_i\subseteq\overline{R_i}
$$
and the standard parabolic facet $F_i$, which is the convex hull of 
$M_i$, is equal to $\overline{R_i} \cap \{x \mid ( x,
\breve{\omega}_{i}) = 1\}$. Since  $\gen \, M_i= \gen \, \Phi$, this
implies that 
$\gen \, R_i = \gen \, \Phi$.
Now the set of regions $R_i$, for all $i\in [n]$, under the action of $W$ covers
all 
the regions $\bigcap_{j=1}^{n+1}\breve{\omega}_{1, j}^{\sigma_j}$ except the two
empty ones.

\par

We have the following consequence of Theorem \ref{coincidono}.

\begin{cor}
\label{biez}
Let $\Phi$ be of type $A_n$. The faces of $\mc P$ are in natural bijection with the nontrivial faces 
of the hyperplane arrangement $\mc H$. Under this bijection, the
$k$-dimensional faces 
of $\mc P$ correspond to the $(k+1)$-dimensional faces of $\mc H$, for
all $k=0, \ldots, n-1$. 
\end{cor}
\begin{proof}
By Theorem \ref{coincidono}, the bijection between the facets of $\mc P$ and the regions of $\mc H$ 
 is given by coning: for all $i \in [n]$, it maps
$F_i$ to $R_i$, which is the cone on $F_i$. Indeed, this map induces the required
bijection since, for all $K \subseteq [n]$, the intersection of the cones on
$F_k$, for $k \in K$, is equal to the cone on the set $\bigcap_{k\in K} F_k$. 
In fact, this implies that the map sending $\bigcap_{k \in K} F_k $ to its cone is
injective and surjective.
\end{proof}

We now give some enumerative results.
\begin{pro}
\label{carat-face}
Let $\Phi$ be of type $A_n$. Then
\begin{enumerate}
\item the face polynomial of $\mc P$ is 
 \label{faccia}
$$
f_\mc P(x)=\sum\limits_{i=0}^{n-1}\binom{n+1}{i+2}(2^{i+2}-2) x^{i},
$$ 
\item 
\label{face}
the face polynomial of $\mc H$ is 
$$
f_\mc H(x)=1+\sum\limits_{i=1}^{n}\binom{n+1}{i+1}(2^{i+1}-2) x^{i},
$$
\item 
\label{carat}
the characteristic polynomial of $\mc H$ is 
$$
\chi_\mc
H(t)= (-1)^n n+\sum\limits_{k=1}^n\binom{n+1}{k+1}(-1)^{n-k}t^k= (-1)^n\sum\limits_{i=1}^n(1-t)^i.
$$
\end{enumerate}
\end{pro}
\begin{proof}
The first equality can be obtained using Theorem \ref{CM} and the computation is left to the reader.

The second equality follows by the first since, 
by Corollary \ref{biez}, the face polynomials of $f_\mc H(x)$ and
$f_\mc P(x)$ satisfy the relation
$f_\mc H(x)=1+xf_\mc P(x)$.

To prove formulas (3), we provide a direct computation of
the M\"{o}bius function of 
$\mc L(\mc H)$.
Since $\mc H$ has $n+1$ hyperplanes and any such $n$ hyperplanes 
are linearly independent, 
$\mc L(\mc H)$ is obtained from the boolean algebra of rank 
$n+1$ by replacing 
all the elements
of both rank 
 $n$ and $n+1$ with a single top element (of rank $n$), the null space
$\underline 0$, whose 
M\"{o}bius function is $\mu(\underline 0)=(-1)^n n$. 
Hence, if $k \geq 1$, the number of $x\in \mc L(\mc H)$ of dimension
$k$ is $\binom{n+1}{k+1}$ and, for all such $x$, 
$\mu(x)=(-1)^{\codim x}$. 
Thus we have:
\begin{equation*}
\begin{split}
\chi_\mc
H(t)&= (-1)^n n+\sum\limits_{k=1}^n\binom{n+1}{k+1}(-1)^{n-k}t^k 
\end{split}
\end{equation*}
and we get the second identity since  
$$
(-1)^n\sum\limits_{i=1}^n(1-t)^i= (-1)^n\sum\limits_{i=1}^n
\sum\limits_{k=0}^i \binom{i}{k} (-t)^k = (-1)^n\sum\limits_{k=0}^n
\sum\limits_{i=\max(k,1)}^n \binom{i}{k} (-t)^k
$$
and $\sum\limits_{i=\max(k,1)}^n \binom{i}{k}= \left\{ \begin{array}{ll}
\binom{n+1}{k+1},  & \textrm{if $k \neq 0$,}\\
n, & \textrm{if $k = 0$.} 
\end{array}  \right.$
%
%
%
%
\end{proof}

By Theorem \ref{zas} and Proposition \ref{carat-face}, (\ref{carat}),  we
re-obtain that $2^{n+1}-2$
is the number of regions of 
$\mc{H}$.
\par

Each vertex of $\mc P$ belongs to exactly $2^{n-1}$ facets of $\mc
P$. 
This is clear if we observe that each vertex $v$ belongs to exactly $n-1$
hyperplanes of $\mc H$ and, if $i$ and $i'$ 
are the two indices such that $(v , \breve\omega_{1,i})$ and 
$(v , \breve\omega_{1,i'})$ are nonzero, then these two values have different
signs. 
It follows that the $2^{n-1}$ sign choices for $\{\breve\omega_{1,j}^\pm\mid
j\ne i, i'\}$
correspond exactly to the regions of $\mc H$ whose closures contain $v$.
\bigskip

We recall the Young diagram combinatorics of the positive roots of type $A_n$, which is useful for representing the abelian ideals and hence the faces. 
The $A_n$ root diagram is a partial $n \times n$ matrix filled with the positive roots: 
the $i$-th row consists of the $n+1-i$ positions from $(i,1)$ to 
$(i,n+1-i)$ (staircase tableau) and the position $(i,j)$ is filled with the
root $\al_{j,n+1-i}$.

\begin{exa}
The following is the $A_4$
root diagram.

$$
\begin{tikzpicture}[scale=.6]
 \draw{(0,0)--(6,0)};
 \foreach \x in {0,1.5, 3, 4.5} 
 \draw{(0,-\x-1.5)--(6-\x,-\x-1.5) 
           (0,0)--(0,-6)
           (6-\x, 0)--(6-\x, -1.5-\x)};
\pgftext[base, left,x=0.45cm,y=-0.8cm]{{$ \al_{1,4}$}};
\pgftext[base, left,x=1.95cm,y=-0.8cm]{{$ \al_{2,4}$}};
\pgftext[base, left,x=3.45cm,y=-0.8cm]{{$ \al_{3,4}$}};
\pgftext[base, left,x=5.1cm,y=-0.8cm]{{$ \al_{4}$}};
\pgftext[base, left,x=0.45cm,y=-2.3cm]{{$ \al_{1,3}$}};
\pgftext[base, left,x=1.95cm,y=-2.3cm]{{$ \al_{2,3}$}};
\pgftext[base, left,x=3.55cm,y=-2.3cm]{{$ \al_{3}$}};
\pgftext[base, left,x=0.45cm,y=-3.8cm]{{$ \al_{1,2}$}};
\pgftext[base, left,x=2.05cm,y=-3.8cm]{{$ \al_{2}$}};
\pgftext[base, left,x=0.55cm,y=-5.3cm]{{$\al_{1} $}};
\end{tikzpicture}
$$ 
 \end{exa} 
The root order corresponds to the 
reverse  order of the matrix positions, 
i.e, if $\be$ fills the position $(i,j)$ and $\ga$ fills the position $(h,
k)$, then $\be\geq \ga$ if and only 
if $i\leq h$ and $j\leq k$ (i.e., if and only if $\be$ is at the north-west of
$\ga$).  
We shall identify  the $A_n$ diagrams and subdiagrams with the corresponding
sets of roots.  
\par

Recall that, for $i=1, \ldots, n$, the standard parabolic facet $F_i$ is the convex hull of $M_i$, which is the dual order ideal generated by the root $\al_i$ in the root poset of $\Phi$: $M_i$ is exactly the set of vertices of $F_i$ and form a maximal rectangle in the tableau.

The realization of the standard parabolic facet 
$F_i$, $i \in [n]$, as the intersection of a closed region 
of $\mc H$ with the affine hyperplane 
$\{(x , \breve\omega_i)=1\}$ shows that $F_i$ has at most
$n+1$ facets. For $i\in \{1,n\}$, $F_i$ has $n$ facets, since $\breve\omega_{i}^\perp \in \mc H$ if (and
only if) $i\in \{1,n\}$. Thus $F_1$ and $F_n$ are $(n-1)$-simplices. 
For $1< i<n$, $F_i$ has in turn $n+1$ facets, each  obtained by  intersecting $F_i$ with a hyperplane $\breve\omega_{1,j}^\perp$, $j\in [n+1]$. 
The facets $F_i\cap \breve\omega_{1,j}^\perp $ of $F_i$ are of two kinds, according to
whether $j\leq i$ or $j>i$. In the first case, $F_i\cap \breve\omega_{1,j}^\perp $ is the convex hull of the roots belonging to $M_i$ but not to its $j$-th column; in the  latter case, it is the convex hull of the roots belonging to $M_i$ but not to its $j$-th row. 
It is immediate to see that these are isometric, respectively,  to the facets of indices $i$
and $i-1$ of the type ${A_{n-1}}$ 
analogue of $\mc P$. Thus $\mc P$ has a recursive structure, in the
sense that its faces of dimension less 
that $n-1$ are the faces of the $A_k$ analogues of $\mc P$, for $k<n$.

\begin{exa}
$$
\begin{tikzpicture}
 \fill[gray!50](0,0) rectangle (2,-2.5);
 \draw{(0,0)--(4,0) (0,0)--(0,-4)};
 \foreach \x in {0.5, 1, 1.5, 2, 2.5, 3, 3.5, 4}
 \draw{(0,-\x)--(4.5-\x,-\x) (\x,0)--(\x,-4.5+\x)};
 \pgftext[base, left, x=1.5cm, y=-4.1cm]{$M_4$ in type   $A_8$};
\end{tikzpicture}
$$

$$
\begin{tikzpicture}
 \fill[gray!50](0,0) rectangle (1,-2.5);
 \fill[gray!50](1.5,0) rectangle (2,-2.5);
 \draw{(0,0)--(4,0) (0,0)--(0,-4)};
 \foreach \x in {0.5, 1, 1.5, 2, 2.5, 3, 3.5, 4}
 \draw{(0,-\x)--(4.5-\x,-\x) (\x,0)--(\x,-4.5+\x)};
 \fill[gray!50](5,0) rectangle (7,-0.5);
 \fill[gray!50](5,-1) rectangle (7,-2.5);
 \draw{(5,0)--(9,0) (5,0)--(5,-4)};
 \foreach \x in {0.5, 1, 1.5, 2, 2.5, 3, 3.5, 4}
 \draw{(5,-\x)--(9.5-\x,-\x) (5+\x,0)--(5+\x,-4.5+\x)};
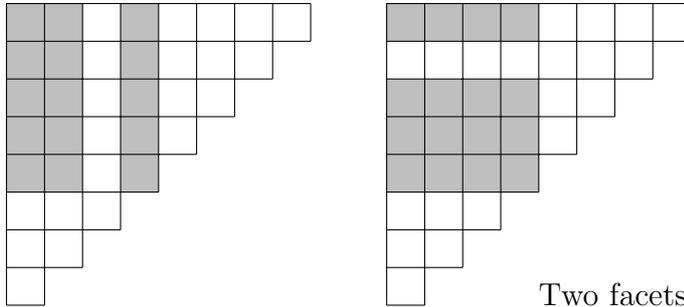
\pgftext[base,left,x=7cm, y=-4cm] {Two facets of $F_4$ in type $A_8$}; 
\end{tikzpicture}
$$

\end{exa}

Moreover, this description of the facets of $F_i$ implies that $F_i$ is
congruent
to the product of two simplices, as already pointed out in \cite{ABHPS} for the
root polytope obtained in the usual coordinate description of $\Phi$ (see,
for example, \cite{Bou}). In particular, $F_i$ is congruent to the product a
$(i-1)$-simplex and a
$(n-i)$-simplex.

Recall from Theorem \ref{CM} that the facets of 
$\mc P$ are obtained from $F_1, \dots, F_n$ through
the action of $W$, that these orbits are disjoint, and that the number of
facets in
the orbit of $F_i$ is
$[W:\stab_W(\breve\omega_i)]=\binom{n+1}{i}$.
On the other hand, the whole automorphism group of $\Phi$ joins the orbits
$WF_i$ and $WF_{n+1-i}$, 
for $i\in [n]$, so that, under this group, $\mc P$ has
exactly $\left\lfloor\frac{n+1}{2}\right\rfloor$ 
orbits of facets, and the $i$-th orbit has cardinality $2\binom{n+1}{i}$, for
$i=1, \dots,\left\lfloor\frac{n+1}{2}\right\rfloor$. 

\begin{exa}
We briefly illustrate our results in type $A_3$, when the root polytope is the
well known \lq\lq cuboctahedron\rq\rq, that is the
intersection of a cube with its dual 
octahedron (see \cite{C}). This has fourteen facets, six of which are squares,
corresponding to the standard parabolic facet $F_2$, and eight regular 
triangles, corresponding to either the facet $F_1$ or the facet $F_3$.
In the next picture, we see a deformation of the polyhedral complex of the three standard parabolic faces, where collinear edges correspond to edges  lying on the same hyperplane of $\mc H$. The four lines of the edges correspond to the 
four hyperplanes of $\mc H$.
\end{exa}

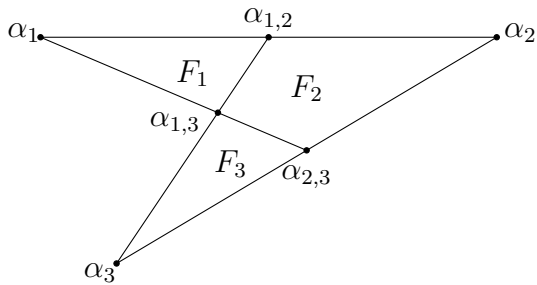
$$
\begin{tikzpicture}
  \filldraw [] (0,0) circle (1pt)
                   (6,0) circle (1pt)
                   (3,0) circle (1pt)
                   (1,-3) circle (1pt)
                   (3.5,-1.5) circle (1pt)
		   (7/3,-1) circle (1pt);
  \draw 
  (0,0) -- (6,0)
  (1,-3) -- (6,0)
  (1,-3) -- (3,0)
  (0,0) -- (3.5,-1.5);
  \pgftext[base, right, x=-.1, y=0]{$\al_1$};
  \pgftext[base, center, x=3cm, y=.2cm]{$\al_{1,2}$};
  \pgftext[base, left, x=6.1cm, y=0cm]{$\al_{2}$};
  \pgftext[top, right, x=1cm, y=-3cm]{$\al_{3}$};
  \pgftext[top, right, x=2.1cm, y=-1cm]{$\al_{1,3}$};
  \pgftext[top, center, x=3.5cm, y=-1.7cm]{$\al_{2,3}$};
  \pgftext[top, center, x=3.5cm, y=-.5cm]{$F_2$};
  \pgftext[top, center, x=2cm, y=-.3cm]{$F_1$};
  \pgftext[top, center, x=2.5cm, y=-1.5cm]{$F_3$};
\end{tikzpicture}
$$

\subsection{Type $C$}
The description of
$\mc P_\Phi$ is much simpler  
and most of the combinatorics 
developed for type $A$ has its analogue. 
Let $\Phi$ be of type $C_n$ and omit
the
subscript $\Phi$ everywhere. 
As noted in \cite{C-M}, the root polytope of every root system  is
the convex hull of its long roots. In our convention, the long positive roots are 
$\lambda_i:= 2 (\sum_{k=i}^{n-1} \al_k) + \al_n$, for $i \in [n]$, and
form a (orthogonal) basis of $\gen \, \Phi$. Then the root
polytope $\mc P$ is a hyperoctahedron (or cross-polytope). 
Its facets correspond one to one to the octants of the cartesian system given by
the basis of the long roots.

As noted in Table \ref{Tarrangement}, 
the arrangement $\mc H$ is given by the orbit $[\breve{\omega}_1^{\perp}]$
of the hyperplane orthogonal to the 
first fundamental coweight 
$\breve{\omega}_1$.
Since the stabilizer of
$\breve{\omega}_{1}$ in the Weyl group $W$ (contragredient representation) is
the parabolic subgroup 
$W\la \Pi \setminus \{ \al_1 \} \ra$,
the orbit of $\breve{\omega}_1$ is
obtained 
by acting with the set of the minimal coset representatives
$W^{1} := \{s_i \cdots s_1 \mid i = 0,
\ldots n-1 \} \cup \{ s_i  \cdots  s_{n-1} s_n \cdots s_1 \mid i = 1,
\ldots n \}$ (for simplicity, we write $s_k$ instead of $s_{\al_k}$, for
$k\in
[n]$).  For
$i \in [1, n]$, let  
$\breve{\omega}_{1,i}:= s_{i-1}\cdots s_1 \left(\breve{\omega}_{1}\right)$. 
We
have 
$$
\breve{\omega}_{1,i} = \left\{ \begin{array}{ll} 
\breve{\omega}_{1}, & \textrm{if $i = 1$,} \\
\breve{\omega}_{i} - \breve{\omega}_{i-1}, & \textrm{if $i \notin \{ 1,n\}$,}
\\                  
2 \breve{\omega}_{n} - \breve{\omega}_{n-1} , & \textrm{if $i = n$.} 
                                  \end{array} \right.
$$

Moreover, $s_n (2 \breve{\omega}_{n} - \breve{\omega}_{n-1}) = - (2
\breve{\omega}_{n} - \breve{\omega}_{n-1})$ and hence the second $n$ functionals
we obtain are the opposites of the first $n$. Thus the hyperplanes of $\mc
H$ are the hyperplanes orthogonal to $\breve{\omega}_{1,i}$, for $i \in [n]$.
Being the $n$ hyperplanes linearly independent, the number of regions of
$\mc H$ is  $2 ^{n}$. In the usual coordinate description of $\Phi$ \cite{Bou}, $\mc H$ is the set of coordinate hyperplanes. 

Hence, Proposition \ref{coni=} implies the following theorem.
\begin{thm}
 If the root system $\Phi$ is of type $C$, the closures of the regions of the
hyperplane
arrangement $\mc H_{\Phi}$ coincide with the cones on the facets of the
root polytope $\mc P_{\Phi}$.
\end{thm}

We recall the Young diagrams combinatorics for the root system of
type $C_n$. To simplify notation, we set $\al_{i,j}:= \sum_{h=i}^{j}
\al_h$, for all $1 \leq i \leq j \leq n$.
The $C_n$ root diagram is a partial $n \times 2n-1$ matrix filled with the
positive roots: 
the $i$-th row consists of the $2(n-i)+1$ positions from $(i,i)$ to 
$(i,2n-i)$; the positions $(i,j)$, with $i\leq j < n$, are filled with the
roots $\al_{j,n-1}+\al_{i,n}$,
the positions $(i,j)$, with $i\leq j = n$,
are filled with the
roots 
$\al_{i,n}$,
and the positions $(i,j)$, with 
$n+1\leq j\leq 2n-i$, are filled with the roots 
$\al_{i,2n-j}$. Note that the long positive roots are in positions $(i,i)$,
$i \in [n]$. 

\begin{exa}
The following is the $C_4$
root diagram.
$$
\begin{tikzpicture}
 \draw{(-4.5,0)--(6,0)};
 \foreach \x in {0,1.5, 3, 4.5} 
 \draw{(-4.5+\x,-\x-1.5)--(6-\x,-\x-1.5) 
           (-4.5+\x,0)--(-4.5+\x,-1.5-\x)
           (6-\x, 0)--(6-\x, -1.5-\x)};
\pgftext[base, left,x=-4.47cm,y=-0.8cm]{\tiny{$\al_{1,3}+\al_{1,4}$}};
\pgftext[base, left,x=-2.95cm,y=-0.8cm]{\tiny{$ \al_{2,3}+ \al_{1,4}$}};
\pgftext[base, left,x=-1.35cm,y=-0.8cm]{\tiny{$ \al_{3}+ \al_{1,4}$}};
\pgftext[base, left,x=0.45cm,y=-0.8cm]{\tiny{$ \al_{1,4}$}};
\pgftext[base, left,x=1.95cm,y=-0.8cm]{\tiny{$ \al_{1,3}$}};
\pgftext[base, left,x=3.45cm,y=-0.8cm]{\tiny{$ \al_{1,2}$}};
\pgftext[base, left,x=5.1cm,y=-0.8cm]{\tiny{$ \al_{1}$}};

\pgftext[base, left,x=-2.95cm,y=-2.3cm]{\tiny{$\al_{2,3}+ \al_{2,4}$}};
\pgftext[base, left,x=-1.34cm,y=-2.3cm]{\tiny{$ \al_{3}+ \al_{2,4}$}};
\pgftext[base, left,x=0.45cm,y=-2.3cm]{\tiny{$ \al_{2,4}$}};
\pgftext[base, left,x=1.95cm,y=-2.3cm]{\tiny{$ \al_{2,3}$}};
\pgftext[base, left,x=3.55cm,y=-2.3cm]{\tiny{$ \al_{2}$}};

\pgftext[base, left,x=-1.34cm,y=-3.8cm]{\tiny{$\al_{3}+ \al_{3,4}$}};
\pgftext[base, left,x=0.45cm,y=-3.8cm]{\tiny{$ \al_{3,4}$}};
\pgftext[base, left,x=2.05cm,y=-3.8cm]{\tiny{$ \al_{3}$}};

\pgftext[base, left,x=0.55cm,y=-5.3cm]{\tiny{$\al_{4} $}};

\end{tikzpicture}
$$ 
 \end{exa}
 
As for type $A_n$, the root order corresponds to the 
reverse order of the matrix positions, 
i.e, if $\be$ fills the position $(i,j)$ and $\ga$ fills the position $(h,
k)$, then $\be\geq \ga$ if and only 
if $i\leq h$ and $j\leq k$. 
We shall identify  the $C_n$ diagrams and subdiagrams with the corresponding
sets of roots.  

\bigskip

In type $C_n$, there is a unique standard parabolic facet $F_n$ and this is
the convex hull of the unique 
maximal abelian ideal $M:=M_n= \{ \al \in \Phi ^+ \mid (\al
, \breve{\omega}_{n}) >0 \}$ (see Section \ref{idealiabeliani}), which is the
dual order ideal generated by the unique long simple root  
$\al_n$.
This is the convex polytope having exactly the positive long roots as its
set of vertices.

$$
\begin{tikzpicture}
[scale=.6]
\label{c8}
\foreach \x in {0,1, 2, 3, 4, 5, 6, 7}
\fill[gray!50](-7+ \x ,0) rectangle (1,-1-\x);
 \draw{(-7,0)--(8,0)};
 \foreach \x in {0,1, 2, 3, 4, 5, 6, 7}
 \draw{(-7+\x,-\x-1)--(8-\x,-\x-1) 
           (-7+\x,0)--(-7+\x,-1-\x)
           (8-\x, 0)--(8-\x, -1-\x)};
\pgftext[base, left, x=-7cm, y=-6cm]{\Large $M$ in type   $C_8$};
\end{tikzpicture}
$$

In the diagram representation of $\Phi^+$ and $M$, the facets 
of $F$ are obtained intersecting with one of the hyperplanes
$\lambda_k^\perp$ orthogonal to the long root $\lambda_k$, $k\in [n]$, and 
hence are the convex hulls of the sets obtained by removing from $M$ the
roots that are also in
$U_{\lambda_k} \cup R_{\lambda_k} \cup \{\lambda_k\}$, 
where, for any long positive root
$\lambda$, we have set
\begin{equation}\label{u-r}
\begin{split}
U_{\lambda}:= \{\be \in \Phi^+ \mid \be - \lambda \in \Phi^+\}\\
R_{\lambda}:= \{\be \in \Phi^+ \mid \lambda - \be \in \Phi^+\}
\end{split}
\end{equation}
In the diagram representation of the root poset, the roots
in $U_{\lambda}$ and $R_{\lambda}$ are obtained from $\lambda$,
respectively, going up and to the right.

$$
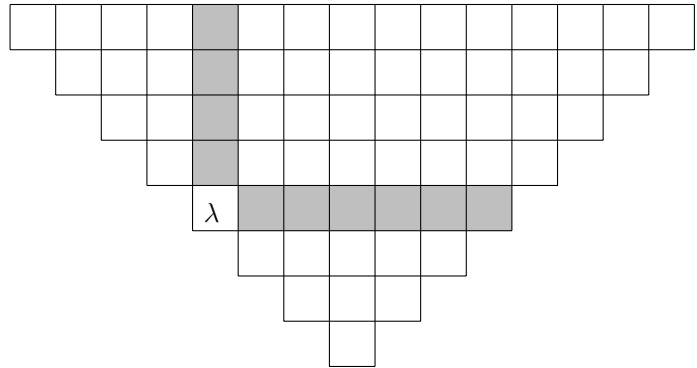
\begin{tikzpicture}
[scale=.6]
\fill[gray!50](-3,0) rectangle (-2,-4);
\fill[gray!50](-2,-4) rectangle (4,-5);
 \draw{(-7,0)--(8,0)};
 \foreach \x in {0,1, 2, 3, 4, 5, 6, 7}
 \draw{(-7+\x,-\x-1)--(8-\x,-\x-1) 
           (-7+\x,0)--(-7+\x,-1-\x)
           (8-\x, 0)--(8-\x, -1-\x)};
\pgftext[base, left,x=-2.75cm,y=-4.8cm]{\Large{$ \lambda$}};
\end{tikzpicture}
$$

\subsection{Types $B$ and $D$}
For $\Phi$ of type $B_n$ or $D_n$, $n\geq 4$, the lattice of regions of
$\mc H$ is
strictly finer than the fan associated to $\mc P$. In fact, some
hyperplanes in $\mc H$ cut some facets of $\mc P$ into two
nontrivial parts.  
We show this for $\mc P$ of type $B_n$; since the embedded $D_n$ made of
the long roots of $B_n$ has the same  polytope, 
we obtain the analogous result also for $D_n$.  
So let $\Phi$ be of type $B_n$, $n\geq 4$. From Table 2, we see that
$\breve\omega_1^\perp$
belongs to $\mc H$. 
\par
First we note that the highest short root $\theta_s$ of $\Phi$ is parallel to $\breve\omega_1$. (Dually this means that the highest root of type $C_n$ is orthogonal to the simple roots $\alpha_i$, $i \in [2,n]$ of type $C_n$,  i.e. is parallel to the first fundamental coweight, and this can be checked from the extended Dynkin diagram of
$C_n$).

Since the short positive roots in $\Phi$ form an orthogonal basis of
$\gen\,\Phi$, the orbit of $\breve\omega_1$ consists of the vectors of an
orthogonal basis of $\gen\,\Phi$ together with their opposites. 
This implies, in particular, that $\real \breve\omega_1$ belongs to the
intersection lattice of  $\mc H$, being the intersection of the
$n-1$ hyperplanes other than $\breve\omega_1^\perp$ in the orbit
$[\breve\omega_1^\perp]$. 
But from \cite[Proposition 3.3]{C-M}, we know that  $\real \breve\omega_1$
contains the
barycenter of
the standard parabolic facet $F_1$ of $\mc P$, therefore each of the
hyperplanes
in $[\breve\omega_1^\perp] \setminus \breve\omega_1^\perp$ cuts $F_1$ in two
nontrivial parts.

\section{Principal maximal abelian ideals of the Borel subalgebra}
\label{idealiabeliani} 
In this section, we first provide some general results on principal abelian
ideals that hold for all irreducible root systems.
Then we shall restrict our attention to the types $A_n$ and $C_n$.
In these cases, the principal abelian ideals corresponding to the standard
parabolic facets are exactly the maximal abelian ideals. 
In these cases, we will construct a triangulation of the standard parabolic facets
related to the poset of the abelian ideals. 
In this section, we state the results, which are formally the same for both types. 
The proofs will be given, separately for the two types,  in Sections \ref{tr-a} and \ref{tr-c}.
\par

For any positive root $\al$, we denote by $M_\al$ the dual order ideal of the root poset generated by $\al$:
$$
M_\al=\{\be\in \Phi^+\mid  \be \geq \al\};
$$
for $i\in[n]$, we set $M_i=M_{\al_i}$, so that  
$$
M_i= \{\be\in \Phi^+\mid (\be, \breve{\omega_i})>0\}.
$$

Recall that we denote by $m_\al$, $\al\in \Pi$, the coordinates of $\theta$ and we set $m_i=m_{\al_i}$, i.e. $m_i=(\theta, \breve\omega_i)$.
\begin{lem}
\label{uno}
The principal ad-nilpotent ideal $M_i$ is abelian if and only if $m_i=1$.
\end{lem}

\begin{proof}
The statement of the lemma is equivalent to the assertion that $M_i$ is abelian if and only if every positive root has 0 or 1 as $i$-th coordinate with respect to the basis given by the simple roots. 
Clearly, if  every positive root has $i$-th coordinate equal to 0 or 1, then the sum of two roots in $M_i$ cannot be a root, and hence $M_i$ is abelian. 
Conversely, by contradiction, let $\be$ be a minimal root in $M_i$ with $i$-th coordinate $\geq 2$: take a simple root $\al$ such that $\be - \al \in \Phi^+$. 
By minimality, $\al = \al_i$ and both $\al_i$, $\be-\al_i$ are in $M_i$. This is a contradiction since $M_i$ is abelian. 
\end{proof}

For instance, in type $A_n$, all ideals $M_i$ are abelian while, in type $C_n$, only  $i_{M_n}$ is.

\begin{lem}\label{unomax}
If $M_i$ is abelian, then it is a maximal abelian ideal.
\end{lem}

\begin{proof}
Let $M_i$ be abelian and let $\ga$ be a maximal positive root not in $M_i$. 
Let $\be \rhd \ga$, i.e. $\beta$ covers $\gamma$ in the root poset. Hence $\be \in M_i$ by maximality and, since the root poset is ranked by the height, $\be = \ga + \al$ for a certain  simple root $\al$. 
Since $\ga \notin M_i$ and $\be \in M_i$, we have that $\al= \al_i$.
Thus we get the assertion since there is no abelian ideal containing both $\ga$ and $\al_i$ since their sum is a root.
\end{proof}

We set
$$
W^i=\{w\in W\mid D_r(w)\subseteq \{\al_i\}\},
$$
where $D_r(w)$ is the set of right descents of $w$. 
It is well known that $W^i$ is the set of minimal length representatives of the left cosets  $W/W\la\Pi \setminus \{ \al_i \} \ra$. 

\begin{lem}\label{Wi-Mi}
Let $w\in W$. Then $w\in W^i$ if and only if $\ov N(w)\subseteq M_i$.
\end{lem}

\begin{proof}
The only simple root in $M_i$ is $\al_i$, hence the claim follows directly from Property 
(\ref{supporto}) and Equation (\ref{discese}). 
\end{proof}

\begin{pro}
\label{vecchiatfae}
Let $i \in [n]$ be such that $m_i=1$, and $N\subseteq M_i$.  
Then $M_i\setminus N$ is an abelian ideal if and only if there exists $w\in  W^{i}$ such that $N=\ov N(w)$. 
\end{pro}
 
\begin{proof}
By Lemma \ref{uno}, $M_i$ is an abelian ideal: in particular the sum of any two roots in $M_i$ is not a root. 
\par

Assume that $M_i\setminus N$ is an abelian ideal. 
By Lemma \ref{Wi-Mi}, it suffices to prove that $N=\ov N(w)$ for some $w\in W$ or, equivalently, that $N$ and $\Phi^+\setminus N$ are both closed. 
Clearly, $N$ is closed because it is contained in $M_i$. 
Suppose  $\be, \be' \in \Phi^+ \setminus N$. 
If both are in $M_i\setminus N$, by the same argument as before $\be + \be'$ is not a root. 
If both do not belong to $M_i$, then also $\be+ \be'$ does not belong to $M_i$ since its $i$-th coordinate is 0, and hence it belongs to $\Phi^+ \setminus N$. 
If one of the roots, say $\be$, is in $M_i \setminus N$, the other is not in $M_i$ and $\be + \be'$ is a root, then $\be + \be' \in M_i$ since its $i$-th coordinate is 1, and hence $\be + \be'$ belongs to $M_i \setminus N$ since this set is a dual order ideal.
\par

Conversely,  assume that $N=\ov N(w)$ for some $w\in W^i$, and let $\be \in M_i\setminus N$,  $\be'>\be$. 
Then $\be'-\be$ is sum of positive roots not in $M_i$, since $(\be, \breve{\omega_i})=(\be' , \breve{\omega_i})=1$. 
Hence $\be'$ is a positive linear combination of roots in $\Phi^+\setminus N$ and thus belongs to $\Phi^+\setminus N$ since this is a convex set.
\end{proof}

\begin{pro}
\label{dispone} 
Fix $i \in [n]$ such that $m_i=1$. Let $w \in W^i$, $\ga, \, \ga' \in M_i$, and $\ga ' \leq \ga$. Then $w(\ga ') \leq w(\ga)$.
\end{pro}

\begin{proof}
If $\ga, \ga' \in M_i$ and $\ga' \leq \ga$, then $\ga- \ga'$ is a sum of simple roots different from $\al_i$.
If $w \in W^i$, $w$ maps all simple roots other than $\al_i$ to positive roots, hence $w(\ga -\ga)=w(\ga)-w(\ga')$ is a sum of positive roots, i.e. $w(\ga')\leq w(\ga)$. 
\end{proof}

\begin{rem}\label{selecta}
By \cite[Remark 7.3]{CMPP}, for all long simple roots $\al$, if $m_\al=1$, then $\max \mc I_\ab(\al)= M_\al$.  
Since in types $A$ and $C$ we have that $m_\al=1$ for all long simple roots $\al$, by Theorem \ref{panyushev} we obtain that the ideals of the form $M_\al$ are the unique maximal abelian ideals, in these cases. 
\end{rem}

By Theorem \ref{CM}, the standard parabolic facets of  $\mc P_\Phi$ are
the convex hulls of the sets $V_{\al_i}=\{\al\in \Phi^+\mid
(\al,\breve{\omega}_i)= m_i\}$ for all simple roots $\al_i$ that do not
disconnect the extended Dynkin diagram of $\Phi$ when removed. By a direct
check, we see that for types $A$ and $C$ such roots $\alpha_i$ are exactly
the long simple roots and $V_{\al_i}=M_i$. Hence, by the above remark, we
obtain the following result.

\begin{thm}\label{facetsAC}
If  $\Phi$ is of type $A_n$ or $C_n$, the set of the standard parabolic facets of $\mc P_\Phi$ is the set  of convex hulls of the maximal abelian ideals of $\Phi^+$.
\end{thm}

In Sections \ref{tr-a} and \ref{tr-c}, we construct a triangulation $\mc T$ of the polytope $\mc P_\Phi$ for $\Phi$ of type $A$ or $C$,  which is 
related to the poset of abelian ideals. The triangulation $\mc T$ contains extra vertices, besides the vertices of $\mc P$. In type $A$, we add only the vertex
$\un 0$; in type $C_n$, we add $\un 0$ and the short roots. In both cases, the set of vertices of any simplex in $\mc T$ is a $\ganz$-basis of the root lattice plus the vertex $\un 0$.

In order to obtain a triangulation of $\mc P_\Phi$, we construct a
triangulation of the standard parabolic facets and transport it to the whole
polytope by the action of $W$. For the types $A_n$ or $C_n$, we should provide a
triangulation of the facets 
$F_\al$, for all long simple $\al$. 

For each abelian ideal $I$ in $\Phi^+$, we define  the {\it border}  $B(I)$ of $I$ as follows:
\begin{equation}\label{bordo}
B(I)=\{\be\in I\mid \gamma_1, \gamma_2\in \Pi, \be-\gamma_1 \text{ and } \be-\gamma_2\in \Phi^+ \Rightarrow \be-\gamma_1-\gamma_2 \not\in I\}.
\end{equation}

\begin{thm}\label{tricomune}
Let $\Phi$ be of type $A_n$ or $C_n$ and let $\al$ be a long simple root in
$\Pi$. Then:
\begin{itemize}
\item[(1)] for all ideals $I\subseteq M_\al$, $\dim(\gen\, I)=n $ if and only if 
$I\in \mc I_\ab(\al)$; 
\item[(2)] for all $I\in \mc I_\ab(\al)$, $B(I)$ is a basis of the root lattice; 
\item[(3)] for all $I\in \mc I_\ab(\al)$, $\{\conv(B(J))\mid J\in \mc I_\ab(\al) \text{ and }J\subseteq I\}$  is a triangulation of $\conv(I)$. 
\end{itemize}
In particular, $\{\conv(B(I))\mid I\in \mc I_\ab(\al)\}$  is a triangulation of~$F_\al$.
\end{thm}

Theorem \ref{tricomune} will be proved in Section \ref{tr-a} for type $A$ and in Section \ref {tr-c} for type~$C$. 

\begin{rem}
If we view  $\wh W$ as a group of affine transformation of the Euclidean space in the usual way \cite{Bou}, then $\wh W_\text{\hskip -2pt\it ab}$ is the set of all elements that map the fundamental alcove $\mc A$ into $2\mc A$ \cite{C-P-Alg}. 
Moreover, if $I$ is an abelian ideal, then there exists $\al\in \Pi$ such that $I \in \mc I_\ab(\al)$ if and only if $w_I(\mc A)$ has a facet on the affine hyperplane $H_{\theta, 2}:=\{x\mid (x, \theta)=2\}$. Indeed, the facet of $\mc A$ orthogonal to $\al$ must be mapped by $w$ into $H_{\theta, 2}$. If $m_\al=1$, then the facet of $\mc A$ orthogonal to $\al$ has the same measure of the facet  of $\mc A$ orthogonal to $\theta$, since there is an element of the extended affine Weyl group that maps one facet into the other
(\cite{IM}, see also \cite{CMoP}). Therefore, since in the $A_n$ and $C_n$ cases $m_\al=1$ for all the long simple roots $\al$, we obtain that $|\bigcup_{\al\in\Pi_\ell}\mc I_\ab(\al)|=2^{n-1}$.  It follows that the total number of simplices that occur in the triangulations of the fundamental facets of $\mc P$ is $2^{n-1}$, in both types.
\end{rem}
\par

Recall that, for any subset $S$ of $\gen \,(\Phi)$, we denote by $\conv_0(S)$ the convex hull of $S\cup\{\un 0\}$. For all long simple roots $\al_i$, we set 

$$
\mc T_i=\{\conv_0(B(I))\mid  I\in \mc I_\ab(\al_i)\}.
$$
Since the stabilizer of the facet $F_i$ is the standard parabolic subgroup
$W\la \Pi\setminus \{\al_i\}\ra$,  
if for all  $\al_i\in \Pi_\ell$ we choose  any set $\mc R_i$ of representatives
of the left cosets of $W\la \Pi\setminus \{\al_i\}\ra$ in $W$  
and set 
$$
\mc R_i\mc T_i=\{wT\mid w\in \mc R_i, \ T\in \mc T_i\},
$$
then 
$$
\bigcup\limits_{\al_i\in \Pi_\ell} \mc R_i\mc T_i
$$ 
is a triangulation of $\mc P_\Phi$.

\begin{thm}\label{indottagen}
Let $\Phi$ be of type $A_n$ or $C_n$. For all $\al_i\in \Pi_\ell$, let $W^i$ be the set of minimal length representatives of the left cosets of $W/W\la \Pi\setminus \{\al_i\}\ra$ and set
$$
\mc T=\bigcup\limits_{\al_i\in \Pi_\ell}  W^i\mc T_i,\qquad
\mc T^+=\{T\in \mc T\mid T\subset \mc P^+\} 
$$
Then $\mc T$ is a triangulation of $\mc P$ and  $\mc T^+$ is a triangulation of $\mc P^+$. In particular, 
$\mc P^+=\mc P\cap \mc C^+,$ where $\mc C^+$ is the positive cone generated by $\Phi^+$.
\end{thm}

\par
It is clear that, in general, the action of the stabilizer of a facet $F$ does not preserve a fixed triangulation of $F$. The choice of the minimal length representatives is essential in Theorem \ref{indottagen}.  
\par
Theorem \ref{indottagen} has a direct application in \cite{Ch}. 
\par
As noted in  \cite{Ch}, $\mc P^+\neq \mc P\cap \mc C^+$ for  all root types other than $A$ and $C$.

\section{Triangulation of $\mc P$: type $A$}
\label{tr-a}
Throughout  this section, $\Phi$ is a root system of type $A_n$.  
Recall that
the positive roots are the roots $\al_{i,j}:=
\sum_{h=i}^{j}
\al_h$, for all $1 \leq i \leq j \leq n$.

\bigskip

As noted in Remark \ref{selecta}, the sets $M_i=\{\beta\in \Phi\mid \beta\geq \alpha_i\}$,  $i\in [n]$, are exactly the maximal abelian ideals of $\Phi$, and $M_i=\max \mc I_\ab(\al_i)$.
In the following proposition, we specialize the result  of Proposition \ref{Ialfa} and determine all the ideals in $\mc I_\ab(\al_i)$, for all $i\in [n]$. 
 
\begin{pro}
Let $I$ be an abelian ideal and $i\in [n]$. Then $I\in \mc I_\ab(\al_i)$ if and only if  $I\subseteq M_i$ and $\{\al_{1,i},\al_{i,n}\}\subseteq I$.  
\end{pro}

\begin{proof}
Let $I$ be an abelian ideal in $\mc I_\ab(\al_i)$.
As noted in Remark \ref{selecta}, $M_i$ is the maximum of 
 $\mc I_\ab(\al_i)$, hence $I\subseteq M_i$. 
\par
Now, we prove that, for all abelian ideals $I$ contained in $M_i$, $w_I^{-1}(-\theta+2\de)\in \Pi_\ell $ if and only if $\{\al_{1,i},\al_{i,n}\}\subseteq I$. 
For all $\be, \ga\in \Phi^+$,  $\be+\ga=\theta$ if
and only if there exists $j\in [n-1]$ such that $\be=\al_{1, j}$ and
$\ga=\al_{j+1, n}$, or vice versa. Since  $\al_{1, j}\geq \al_{1,
i}$ for $i\leq j$ and $\al_{j+1, n}\geq \al_{i, n}$ if $i>j$, it follows
that if $\{\al_{1,i},\al_{i, n}\}\subseteq I$, then  for all $\be, \ga\in
\Phi^+$ such that  $\be+\ga=\theta$, exactly one of $\be$, $\ga$
belongs to $I$. Conversely, if $\{\al_{1,i}, \al_{i, n}\}\not\subseteq I$,
there exists at least one decomposition of $\theta$ as a sum of positive roots
$\be$ and $\ga$ with $\be\not\in I$ and $\ga\not\in I$. Hence, the claim follows from Proposition \ref{Ialfa}.
\par
Thus, we have proved the \lq\lq only if\rq\rq\ part.  Moreover, we have proved that if
$I\subseteq M_i$ and $\{\al_{1,i},\al_{i,n}\}\subseteq I$, then $I\in
\mc I_\ab(\al)$ for some $\alpha\in \Pi_\ell$. 
By Proposition \ref{panyushev}, this forces $I\in\mc I_\ab(\al_i)$ since, if $j\neq i$, the  elements in  $\mc I_\ab(\al_j)$ and those in $\mc I_\ab(\al_i)$ are pairwise incomparable and  hence no ideal in $\mc I_\ab(\al_j)$ can be contained in $M_i$. 
\end{proof}

\bigskip

For all $\al, \be\in \Phi^+$, by a {\it path from $\al$ to $\be$} we mean a sequence $(\al=\be_1, \be_2, \dots,$ $\be_k=\be)$ such that, for $i=1, \dots, k-1$, $\be_i$ covers, or is covered by, $\be_{i+1}$ in the root poset, i.e. $\be_i-\be_{i+1}\in \pm \Pi$. 
For convenience sake, we will sometimes identify a path with its underlying set.
A {\it minimal path from $\al$ to $\be$} is a path of minimal length among all paths from $\al$ to $\be$.

$$
\begin{tikzpicture}
 \fill[gray!20](0,0) rectangle (2,-2.5);
 \fill[gray!80](0,-2) rectangle (.5,-2.5);  
 \fill[gray!80](.5,-1) rectangle (1,-2.5);
 \fill[gray!80](1,0) rectangle (1.5,-1.5); 
 \fill[gray!80](1.5,0) rectangle (2,-.5);
 \draw{(0,0)--(4,0) (0,0)--(0,-4)};
 \foreach \x in {0.5, 1, 1.5, 2, 2.5, 3, 3.5, 4}
 \draw{(0,-\x)--(4.5-\x,-\x) (\x,0)--(\x,-4.5+\x)};
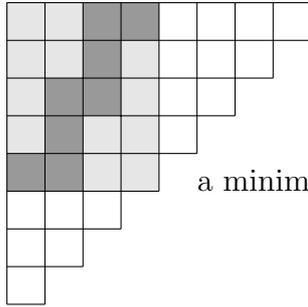
\pgftext[base,left, x=2.5cm, y=-2.5cm] {a minimal path from 
$\al_{1,4}$ to $\al_{4,8}$ in type $A_8$};
\end{tikzpicture}
$$
The minimal paths from $\al_{1,i}$ to $\al_{i,n}$ are in natural bijection with the
abelian ideals that contain both $\al_{1,i}$ and $\al_{i,n}$, i.e. with the
elements in  
$\mc I_\ab(\al_i)$.
The bijection associates to each minimal path $B$ the dual order ideal
$I(B)$ that  it generates 
in $\Phi^+$, i.e.
$$
I (B):=\{\al\in \Phi^+\mid \al\geq \be\ \text{for some}\
\be\in B\}
$$
and, conversely, to each abelian ideal $I $ containing  $\al_{1,i}$
and $\al_{i,n}$ we associate its {\it border} 
$$
B(I )=\{\al_{s,t}\in I \mid \al_{s+1,t-1}\notin I \}.
$$ 
This is the specialization to $A_n$ of the definition  given in (\ref{bordo}).
\par

It is clear that, for all $i\in [n]$,  any minimal path $B$ from $\al_{1,i}$ to 
$\al_{i,n}$ contains $n$ roots.  Moreover, the roots  in $B$ are linearly
independent, since the set of differences between two adjacent roots is
$\Pi\setminus\{\al_i\}$, and any root in $B$ has coefficient $1$ in
$\al_i$.  Thus, the roots in $B$ are a basis of $\gen\, \Phi$. Since $\Phi$ is
of type $A_n$, this implies that they are  a $\ganz$-basis of the root lattice.
This is a well known fact; we prove it here for completeness.

\begin{pro}
\label{unimodulare}
 Every vector basis of $\gen\, \Phi$ contained in $\Phi$ is a basis
of the root lattice.
\end{pro}
\begin{proof}
We proceed by induction on $n$, the case $n=1$ being trivial.
Let $n >1$ and $\be_1, \ldots \be_n$ be linearly independent vectors in
$\Phi$: we may clearly assume that they are in $\Phi^+$. Suppose that there is a
unique $i\in [n]$ such that $\be_i > \al_n$ (i.e. $c_n(\be_i) \neq 0$):
then, by induction hypothesis, $\{\be_1, \ldots, \be_n\} \setminus
\{\be_i\}$ gives a basis of the lattice generated by $\Pi \setminus
\{\al_n\}$. Since $c_n(\be_i)=1$, we get the assertion.
\par

Suppose now that the vectors of the
basis which are greater than $\al_n$ are $\be_{i_1}, \ldots
\be_{i_s}$, with $s >1$. Since $\Phi$ is of type $A_n$,
these roots form a chain in the root poset: we assume that
$\be_{i_1}> \be_{i_2} >  \cdots >
\be_{i_s}$. Then $$\{\be_t \mid t \in [n]\setminus\{ i_1, \ldots, i_{s-1} \}\}
\cup \{\be_{t} - \be_{i_s} \mid t \in \{i_1, \ldots, i_{s-1} \} \}
$$ is a
vector basis contained in $\Phi^+$, generating the same lattice as $\{\be_1, \ldots \be_n\}$, but with a  unique element
greater than
$\al_n$. Hence we may conclude applying the previous argument.
\end{proof}

The set $\mc I_\ab(\al_i)$ parametrizes a triangulation of $F_i$.

\begin{pro}\label{tri-i}
Let $i \in [n]$. The set 
$$
\mc T'_i=\{\conv(B(I))\mid I\in \mc I_\ab(\al_i)\}
$$
is a triangulation $F_i$. Two simplices are adjacent
in  $\mc T'_i$ if and only if the corresponding abelian ideals differ in only 
one element.
\par

In particular, for every triangulation of the standard parabolic facet $F_i$ whose vertices are the roots in $F_i$,  the number of simplices equals the cardinality of $\mc I_\ab(\al_i)$. 
\end{pro}

\begin{proof}
By Proposition \ref{unimodulare}, all triangulations whose vertices are the roots have the same number of
simplices. Moreover, as already noted, every abelian ideal
in $\mc I_\ab(\al_i)$ is uniquely determined by its border. Hence the
first statement implies the second one.
\par

To prove the first statement, recall that the standard parabolic facet $F_i$ is
congruent to the product of
two simplices and that the borders of the ideals we  are
considering coincide with the minimal paths from 
$\al_{1,i}$  to $\al_{i,n}$ in the rectangle $M_i$ corresponding to $F_i$.
The triangulation $\mc T'_i$ is
the so called ``staircase triangulation'' of the product of two simplices 
(see \cite{D-R-S}, \S 6.2.3) and satisfies the required property.
\end{proof}

The stabilizer in $W$ of the face $F_i$ of $\mc P$ is the parabolic
subgroup $W \la \Phi \setminus \al_i \ra$. Let $W^i$ be  
the set of the minimal length representatives of its left cosets, which
corresponds to the orbit of $F_i$
under $W$.
Through the action of the elements in $W^i$, for all $i\in [n]$, we can
induce, from the triangulation of  the standard parabolic facets $F_i$, a triangulation of the whole  $\mc P$. Hence we obtain the following result.

\begin{thm}
\label{tri}
For all $i \in [n]$, let 
$$
\mc T_i=\{\conv_0(B(I)) \mid I\in \mc I_\ab(\al_i)\}\quad
\text{and}
\quad\mc T=\bigcup_{i\in [n]} W^i\mc T_i. 
$$
Then $\mc T$ is a triangulation of $\mc P$. 
\end{thm}

\bigskip

The triangulation $\mc T$ of $\mc P$ is parametrized by the set 
$$
\{(i, w, T) \mid i\in[n], w \in W^i, T\in \mathcal T_i\}.
$$
In particular, as already noted in \cite{ABHPS}, $\mc T$  has
$\sum_{i=1}^{n}|W^i||\mathcal T_i|= \sum_{i=1}^{n} 
\binom{n+1}{i}\binom{n-1}{n-i}= (n+1)C_n$
simplices, where $C_n=\frac{1}{n+1}\binom{2n}{n}$ is the $n$-th Catalan number.
\par

Actually, the triangulation $\mathcal T$ induces a triangulation  of the
positive root polytope $\mathcal P^+$.

\begin{thm}
\label{indotta}
Let
$\mathcal C^+$ denote the positive cone generated by $\Pi$. Then $\mathcal P^+=\mathcal
C^+\cap
\mathcal P$
and the triangulation ${\mathcal T}$  of $\mathcal P$ restricts to a
triangulation of $\mathcal P^+$.
\end{thm}

\begin{proof}
We shall prove the following two statements, which give the result:
\begin{enumerate}
 \item the triangulation ${\mathcal T}$  restricts
to a
triangulation of $\mathcal C^+\cap
\mathcal P$,
\item $\mathcal
P^+=\mathcal C^+\cap
\mathcal P$.
\end{enumerate}
The first assertion is equivalent to requiring that every simplex of the triangulation $\mathcal T$ is either contained in $\mathcal C^+$, or intersects it in a null set. 
Hence, we must show that, for any  $i \in [n]$, $w\in W^i$, and $T\in \mathcal T_i$, either $w(T)\subseteq \mathcal C^+$, or $w(T)\cap \mathcal C^+$ has volume equal to $0$.
\par

Let $i \in [n]$, $w\in W^i$, $T\in\mathcal T_i$, $T=\conv_0 (B(I))$  with $I$
abelian ideal in $M_i$, such that $w(T)\not\subseteq \mc C^+$. Then,
$w(B(I))\not\subseteq \Phi^+$, hence there exists $\gamma\in B(I)$ and $t \in
[n]$ such that $(w(\gamma), \breve\omega_t)<0$. We need to prove that, for all
$\gamma'\in B(I)$, we have that $(w (\gamma') , \breve\omega_t)\leq 0$ (so that
the hyperplane $\breve\omega_t^{\perp}$ separates $w(T)$ and $\mathcal C^+$).

By definition of $B(I)$, 
 there exists $\gamma''\in M_i$ such that $\gamma''-\gamma\in \Phi^+\cup \{\un 0\}$ and $\gamma''-\gamma'\in \Phi^+\cup \{\un 0\}$.   
By Proposition \ref{dispone}, it suffices to prove that  $(w (\gamma'') , \breve\omega_t)\leq 0$. If $\gamma'' = \gamma$, we are done. Otherwise, since both $\gamma$ and $\gamma''$ are in $M_i$, $\gamma'' - \gamma$ is a positive root not in $ M_i$ and hence $w (\gamma'' - \gamma) \in \Phi^+$. It follows that $(w (\gamma'') , \breve\omega_t) = (w (\gamma''- \gamma) , \breve\omega_t) + (w (\gamma) , \breve\omega_t) =  (w (\gamma''- \gamma) , \breve\omega_t) - 1 \leq  0$, since $m_t = 1 $ for all $t\in [n]$ and thus $|(\alpha , \breve\omega_t)| \leq 1 $ for all $\alpha \in \Phi$.

\par

It remains to prove the second assertion. 
The inclusion $\mathcal P^+\subseteq\mathcal C^+\cap \mathcal P$ is obvious, since the origin and positive roots are contained in both the convex sets $\mathcal C^+$ and $\mathcal P$. We have to prove the reverse inclusion.
By the first statement, $\mathcal C^+\cap \mathcal P$ is union of simplices in ${\mathcal T}$. 
Since $ \mathcal C^+ \cap \Phi^- = \emptyset$, the vertices of such simplices are in $\Phi^+ \cup \{0\}$, and hence in $\mc P^+$.
\end{proof} 

As a corollary, we obtain the fact that   
the triangulation $\mc T'_i$ of Proposition \ref{tri-i} inherits the poset structure of $\mc I_\ab(\al_i)$. 

\begin{cor}
Let $I\in \mc I_\ab(\al_i)$ be an abelian ideal. Then the set 
$$
\mc T'_I=\{\conv(B(J))\mid J\in \mc I_\ab(\al_i),\ J\subseteq I\}
$$ 
is a triangulation of $\conv(I)$. 
\end{cor}

\begin{proof}
Let $w\in W^i$ be such that $I=M_i\setminus \ov N(w)$ (Proposition
\ref{vecchiatfae}). Then $w(I)\subset \mc P^+$ and $w(\al)\not\in\mc P^+$ for
all $\al\in M_i\setminus I$.  
By Theorem \ref{indotta}, there exists a subset $\mc S_w$ of $\mc T'_i$ such
that $w(F_i)\cap \mc P^+=\bigcup \{wT\mid T\in \mc S_w\}$. Since $F_i=\bigcup
\{\conv(B(J))\mid  J\in \mc I_\ab(\al_i)\}$, it  must be that 
$\mc S_w=\{\conv(B(J))\mid  J\in \mc I_\ab(\al_i),\ J\subseteq I\}$. The claim follows.
\end{proof}

\begin{rem}
It is clear that  $\mc P^+\subseteq \mc P
\cap \mc C^+$, but for a general irreducible $\Phi$, the equality may not hold.
It is immediate that the equality does not hold for $\Phi$ of type $G_2$. We give a
counterexample for $\Phi$ of type $B_3$. 
By Theorem \ref{CM}, the simplex generated by $\al_2+2\al_3$,
$\al_1+\al_2+2\al_3$, $\al_1+2\al_2+2\al_3$ is a standard
parabolic facet of $\mc P$. 
It is transformed under $s_2s_3s_2$ to the simplex generated by $-\al_2,
\al_1+\al_2+2\al_3, \al_1$, therefore this simplex is a facet. 
It follows that $\frac{1}{2}( -\al_2+
\al_1+\al_2+2\al_3)=\frac{1}{2}( \al_1+2\al_3)$ belongs to
the boundary of $\mc P$. 
But the convex linear combinations of $\al_1$ and $\al_3$ belong to
the boundary $\mc P^+$, therefore the line through $\frac{1}{2}(
\al_1+2\al_3)$ and the origin cuts $\mc P^+$
in $\frac{1}{3}( \al_1+2\al_3)$.
\end{rem}

\section{Directed graphs and simplices}
\label{digraphs}
The triangulation obtained for $\mc P^+$ in the $A_n$ case is the triangulation given by the {\it anti-standard bases} described in \cite{GGP}.  This can be represented as a special set of trees. We can extend this interpretation to the whole triangulation of $\mc P$. 
\par

Let us consider the following class of simple directed graphs (no loops, no
multiple edges). 
Given a directed edge $e$, we write $e=(s,t)$ if $s$ and $t$ are, respectively,
the source and the target of $e$.
We call a simple directed graph \emph{$n$-anti-standard} if it has
$[n+1]$ as
vertex set and exactly $n$ directed edges, 
and satisfies the 
following properties:
\begin{enumerate}
\item every vertex is either a source or a target, but not both (abelianity);
\item for all edges $e=(s,t)$ and $e'=(s',t')$, if $s<s'$ then $t\leq t'$. 
\end{enumerate} 
We can make the arcs correspond to the roots in $\Phi$ in this way: to
the arc $(i,j)$ with  $i < j$, we associate the positive root 
$\al_{i,j-1}$, and to the reversed arc the opposite root,   $-\al_{i,j-1}$.
So, if the vertices $1,2, \ldots, n+1$ lie on a horizontal line, from left to right in
the natural order,  the positive roots are exactly the arcs going from left to right, and among these, the simple roots are the arcs between adjacent vertices.  
With this correspondence, an $n$-anti-standard graph corresponds to a subset
of $\Phi$ of cardinality $n$. 
The first of the above conditions says that the corresponding set of roots is
{\it abelian}, in the sense that for any pair 
of roots in the set, their sum is not a root. 
The second condition implies that for any two positive comparable roots in the set, 
their difference is a root. 
\par

The anti-standard directed graphs generalize the concept of anti-standard tree 
introduced in \cite{GGP}. There, the authors consider the
positive root polytope $\mc P^+$ associated with the root system $A_n$ in
the usual coordinate description (see, for example,  \cite{Bou}). In
the coordinate description, $\gen \, \Phi$ is the subspace of $\mathbb
R ^{n+1}$ orthogonal to $\sum_{i=1}^{n+1} \varepsilon_i$ and the positive root
$\al_{i,j}$ is $\varepsilon_i - \varepsilon_{j+1}$, for all $i,j\in [n]$,
where $\varepsilon_1, \ldots, \varepsilon_{n+1}$ is the standard basis of
$\mathbb R ^{n+1}$. The edge $e=(h,k)$ corresponds to the root $\varepsilon_h -
\varepsilon_k$.

\begin{lem}
\label{anti}
Let $i \in [n]$. The $n$-anti-standard graphs such that the vertices  $1, \ldots, i$ are sources
and the vertices $i+1, \ldots, n+1$ are targets correspond to the simplices of the
triangulation $\mc T_i'$ of the facet $F_i$.
\end{lem}

\begin{proof}
Recall from Proposition \ref{tri-i} that the simplices of  the triangulation $\mc T_i'$ are the convex hulls of the borders of the ideals in $\mc I_\ab(\al_i)$, i.e. the convex hulls of the minimal paths from $\al_{1,i}$  to $\al_{i,n}$. 
Let $p$ be such a path: it contains $n$ roots and it is contained in the rectangle $M_i$, which is the set $\{ \al_{h,k} \mid h\in [i], k \in [i,n] \}$. 
Hence the corresponding graph has $n$ edges,  the vertices $1, \ldots, i$ as sources and the vertices $i+1,\ldots, n+1$ as targets.
Moreover, it satisfies Property (2) of the definition of $n$-anti-standard graph since it corresponds to a path.
\par

Conversely, let $G$ be an $n$-anti-standard graph such that $1, \ldots, i$ are
sources and $i+1,\ldots, n+1$ are targets. 
For all $s \in [i]$, let $T_s$ be the set of all the targets of the edges with $s$ as source.
By Property (2) of the definition of $n$-anti-standard graph  we have that $\max T_s \leq \min T_ {s+1}$, for all $s\in [i-1]$. Hence,
\begin{enumerate}
\item $T_s$ is an interval in $ [i+1,n+1]$, for all  $s \in [i]$, 
\item $|T_1\cup\dots\cup T_i|=|T_1|+\cdots+|T_i|-(|T_1\cap T_2|+\cdots +|T_{i-1}\cap T_i|)$, 
\item $|T_s\cap T_{s+1}|\leq 1$  for $1\leq s < i$. 
\end{enumerate}
By assumption,  $|T_1\cup\dots\cup T_i|=n+1-i$, while $|T_1|+\cdots+|T_i|=n$, since $n$ is the number of arcs in $G$. It follows that $|T_s\cap T_{s+1}|=1$, i.e. that  $\max T_s = \min T_ {s+1}$, for all $s\in[i-1]$.
Hence $G$ corresponds to a minimal path from $\al_{1,i}$ to $\al_{i,n}$.
\end{proof}

\begin{thm}
\label{digra}
The simplices of the triangulation $\mc T'$ of the border of $\mc P$ are exactly the sets corresponding to the $n$-anti-standard graphs.
\end{thm}

\begin{proof}
By Theorem \ref{tri} and Lemma \ref{anti}, we need to
show that each $n$-anti-standard graph is the image, under some $w\in W^i$, of a  $n$-anti-standard graphs such that $1, \ldots, i$ are sources and $i+1, \ldots, n+1$, for some $i\in [n]$. 
Using the coordinate description, $W$ acts as the group of permutations of the
standard basis vectors $\varepsilon_1, \ldots, \varepsilon_{n+1}$
and this action induces a 
faithful action on the set of digraphs of vertex set $[n+1]$. It is clear that
this action preserves abelianity. 
The permutations belonging to $W^i$ are exactly the shuffles of the first $i$
nodes with the remaining $n+1-i$ ones, i.e. 
they are all the permutations $\sigma$ such that $\sigma(1)<\cdots<\sigma(i)$
and
$\sigma(i+1)<\cdots< \sigma(n+1)$. Hence this action preserves also the property
of being an  $n$-anti-standard graph.
\par

On the other hand, let $G$ be an $n$-anti-standard graph and let $i$ be
the number of
sources of $G$ (so that $n+1-i$ is the number 
of targets by abelianity). Consider the permutation 
$\sigma$ which moves all sources in the first $i$ position without changing the
relative order among the sources and among 
the targets. Then $\sigma(G)$ corresponds to a minimal path from $\al_{1,i}$
to $\al_{i,n}$ and $\sigma^{-1}\in W^i$. 
We get the assertion.
\end{proof}
In the next figure we show an anti-standard graph and
the distinguished one which
it is obtained from. 
$$
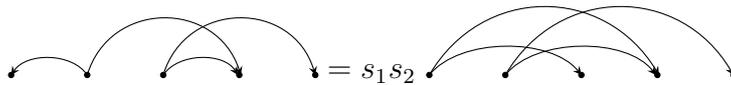
\begin{tikzpicture}
 \foreach \x in {1, 2, 3, 4, 5}
 \filldraw (\x,0) circle(1pt);
 \pgfsetarrowsend{stealth}
 \draw  (2,0)..controls(2.3,1)and(3.7,1)..(4.02,0);  
 \draw  (3,0)..controls(3.3,1)and(4.7,1)..(5,0);
 \draw  (3,0)..controls(3.2,.3)and(3.8,.3)..(3.98,0);
 \draw  (2,0)..controls(1.8,.3)and(1.2,.3)..(1,0);
\end{tikzpicture} =
s_1s_2\ \begin{tikzpicture}
 \foreach \x in {1, 2, 3, 4, 5}
 \filldraw (\x,0) circle(1pt);
 \pgfsetarrowsend{stealth}
 \draw  (1,0)..controls(1.4,.5)and(2.6,.5)..(3.,0);  
 \draw  (1,0)..controls(1.7,1.2)and(3.3,1.2)..(4.02,0);
 \draw  (2,0)..controls(2.4,.5)and(3.6,.5)..(3.98,0);
 \draw  (2,0)..controls(2.7,1.2)and(4.3,1.2)..(5,0);
\end{tikzpicture} 
$$
The graph on  the right side of the above picture corresponds
to the following minimal path in 
$M_2$, for type $A_4$.  
$$
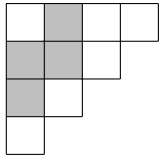
\begin{tikzpicture}
 \fill[gray!50](0,-.5) rectangle (.5,-1.5)  (.5,0) rectangle (1,-1);
 \draw{(0,0)--(2,0) (0,0)--(0,-2)};
 \foreach \x in {0.5, 1, 1.5, 2, 2.5}
 \draw{(0,-\x)--(2.5-\x,-\x) (\x,0)--(\x,-2.5+\x)};
\end{tikzpicture}
$$

The anti-standard graphs whose associated sets of roots are contained
in $\Phi^+$ correspond to the 
anti-standard trees in \cite{GGP} (see also \cite[ex. 6.19-q.]{StaEC2}): in
fact,
the restricted triangulation of 
Theorem \ref{indotta} is the triangulation of $\mc P^+$ 
given by the {\it anti-standard bases} studied in [loc. cit.]. 
\par

We could have done the analogous construction replacing, for all $i\in [n]$, 
the minimal paths from $\al_{1,i}$ to $\al_{i,n}$ with the minimal paths 
from $\al_i$ to the highest root $\theta$.  We notice that the minimal paths from $\al_i$ to $\theta$ correspond to the standard bases of \cite{GGP} that are included in $M_i$ and,  in fact, they still yield a triangulation of $F_i$. The triangulation obtained in this way is the transformed one of the previous one by an element  of the Weyl group. Indeed, looking at the diagram combinatorics, we can easily find an element in the Weyl group that induces a bijection  between the sets of  the standard and anti-standard bases.
For example, we can take the involution
$$
w_i^c:=s_{\al_{i_*,i_*'}}\cdots s_{\al_{2,i-2}}  s_{\al_{1,i-1}}
$$
where $i_*=\lfloor i/2\rfloor $, $i_*'=\lfloor (i+1)/2\rfloor$ (integral parts).
By a direct check, we can see that $w_i^c$ acts on $M_i$ as the antipodal
permutation of the columns, hence it maps $\al_{1,i}$ to $\al_i$, 
$\al_{i,n}$ to $\theta$, and trasforms the minimal paths from  $\al_{1,i}$
to
$\al_{i,n}$ into the minimal paths from $\al_i$ to $\theta$.
Hence, $w^c$ preserves $F_i$ and  transforms the triangulations of $F_i$  associated to the two kinds of paths into each other. 

\section{Triangulation of $\mc P$: type $C$}
\label{tr-c}
Throughout this section, $\Phi$ is a root system of type $C_n$. Recall that the only
long simple root is $\al_n$, the positive short roots are the roots $\al_{i,j}=
\sum _{k=i}^{j} \al_k$, for all $1\leq i\leq j\leq n$ except $i=j=n$, and  $\alpha_{i,n}+\alpha_{j, n-1}$,  for all $1\leq i< j\leq n-1$,
while the positive long roots are $\lambda_n=2(\sum_{k=i}^{n-1} \al_k) +
\al_n$, for $i \in [n]$.
\par

If $\be$ is any root not in $M:=M_n=\{\alpha \in \Phi^+ \mid c_n(\alpha)=1\}$,
then both $\ga= \al_1 + \cdots + \al_n $ and $\ga'= \al_1 + \cdots + \al_{n-1} $
are greater than or equal to $\be$, and $\ga + \ga'$ is a 
root. Hence all the abelian ideals are contained in $M$, and $M$ is the
unique maximal  abelian ideal. By Theorem \ref{CM}, $\mc P$ has $F:=F_n$ as its
unique standard parabolic facet; the stabilizer of $F$ is the parabolic
subgroup $W \la \Phi \setminus \al_n \ra$, and the set of all facets of $\mc P$
is  the orbit of $F$ under the action of $W$.
\par

As a matter of fact, the
triangulation we shall construct has as vertices not only the vertices of $\mc
P_{\Phi}$, which are the long roots of $\Phi$, but also the short roots of
$\Phi$, which lie in the inner part of edges of $\mc
P_{\Phi}$ (see \cite[\S 5]{C-M}). 
\par

Recall that $W^n$ denotes   
the set of the minimal left cosets representatives of $W \la \Phi \setminus
\al_n \ra$, which corresponds to the 
orbit of $F$ under $W$.

\begin{pro}
Let $I$ be an abelian ideal. Then $I\in \mc I_\ab(\al_n)$ if and only if 
$\al_{1,n}\in I$. 
\end{pro}

\begin{proof}
Let $\theta=\beta+\gamma$ be a decomposition of $\theta$ as a sum of two
positive roots. 
Then exactly one of $\be$ and $\ga$ belongs to $M$. Assume $\be \in M$, so
that  $\ga\in \Phi(\Pi\setminus\{\al_n\})$. Since
$\theta=2\al_1\cdots+2\al_{n-1}+\al_n$ and since $\Phi(\Pi\setminus\{\al_n\})$
is of type $A_{n-1}$, we obtain that $\be\geq \al_{1,n}$. Since
$\theta=\al_{1,n}+\al_{1, n-1}$, we get the assertion by  Proposition \ref{Ialfa}.
\end{proof}

The  border strips $B(I)$ of the ideals in $\mc I_\ab(\al_n)$ are the minimal 
paths from $\al_{1,n}$ to a long root. All such minimal paths consist of $n-1$
steps, either leftwards, or downwards,  starting from $\al_{1,n}$, all choices
being allowed. In particular, they are $2^{n-1}$, in number, and hence $|\mc
I_\ab(\al_n)|=2^{n-1}$.

$$
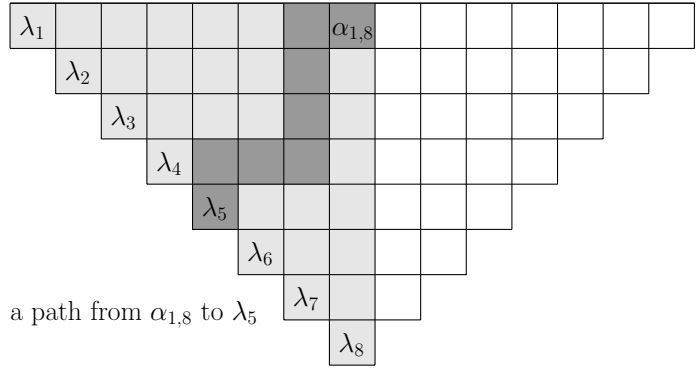
\begin{tikzpicture}
[scale=.6]
\label{path_c}

\foreach \x in {0,1, 2, 3, 4, 5, 6, 7}
\fill[gray!20](-7+ \x ,0) rectangle (1,-1-\x);
 \draw{(-7,0)--(8,0)};

\fill[gray!80](-1 ,0) rectangle (1,-1);

\fill[gray!80](-1 ,-1) rectangle (0,-3);

\fill[gray!80](-3 ,-3) rectangle (0,-4);

\fill[gray!80](-3 ,-4) rectangle (-2,-5);

\draw{(-7,0)--(8,0)};
 \foreach \x in {0,1, 2, 3, 4, 5, 6, 7}
 \draw{(-7+\x,-\x-1)--(8-\x,-\x-1) 
           (-7+\x,0)--(-7+\x,-1-\x)
           (8-\x, 0)--(8-\x, -1-\x)};
\pgftext[base, left, x=.05cm, y=-.7cm]{\Large $\al_{1,8}$};
\pgftext[base, left, x=-6.8cm, y=-.7cm]{\Large $\lambda_1$};
\pgftext[base, left, x=-5.8cm, y=-1.7cm]{\Large $\lambda_2$};
\pgftext[base, left, x=-4.8cm, y=-2.7cm]{\Large $\lambda_3$};
\pgftext[base, left, x=-3.8cm, y=-3.7cm]{\Large $\lambda_4$};
\pgftext[base, left, x=-2.8cm, y=-4.7cm]{\Large $\lambda_5$};
\pgftext[base, left, x=-1.8cm, y=-5.7cm]{\Large $\lambda_6$};
\pgftext[base, left, x=-.8cm, y=-6.7cm]{\Large $\lambda_7$};
\pgftext[base, left, x=.2cm, y=-7.7cm]{\Large $\lambda_8$};

\pgftext[base, left, x=-7cm, y=-7cm]{\Large a path from $\al_{1,8}$ to 
$\lambda_5$ };
\end{tikzpicture}
$$

\begin{pro}\label{unimodulareC}
For all abelian ideals $I\in \mc I_\ab(\al_n)$, $B(I)$ is a basis of the root lattice. 
\end{pro}

\begin{proof} 
Two adjacent roots in a minimal path from $\al_{1, n}$ to a long root in the root diagram differ by a
simple root;  the set of the positive differences between adjacent roots in a
fixed path is $\{\al_1, \dots, \al_{n-1}\}$; since $\al_{1, n}$
contains $\al_n$ with coefficient $1$, all the simple roots are integral
linear combinations of the roots in the path.  
\end{proof}

This set parametrizes a triangulation of $F$.

\begin{pro}\label{tri-n}
The set 
$$
\mc T'_n=\{\conv(B(I))\mid I\in \mc I_\ab(\al_n)\}
$$
is a triangulation of $F$. 
\par

Two simplices are adjacent
in  $\mc T'_n$ if and only if the corresponding abelian ideals differ in only
one element.
\end{pro}

\begin{proof}
We have to prove that $F=\bigcup\limits_{ I\in \mc I_\ab(\al_n)}\conv(B(I))$
and that, 
for all $I, I'\in  \mc I_\ab(\al_n)$, $\conv(B(I))\cap
\conv(B(I'))$ is a common face of $\conv(B(I))$ and $\conv(B(I'))$.
\par
It is easy to check that $F$ is the convex hull of the long positive roots $\lambda_i$. Therefore, in order to prove that
$F=\bigcup\limits_{ I\in \mc I_\ab(\al_n)}\conv(B(I))$,  it suffices to prove
that any positive linear combination of long positive roots is a positive linear
combination of the roots in $B(I)$, for some $I$ in $\mc I_\ab(\al_n)$. 
\par

Let $x=\sum_{i=1}^n x_i\lambda_i$ with $x_i\geq 0$. We will find the
appropriate $I \in \mc I_\ab(\al_n)$ step by step. Since $\lambda_1 + \lambda_n
=
2\al_{1,n}$, we have:
\begin{itemize}
\item[(a)]
if $x_1>x_n$, then $x=(x_1-x_n)\lambda_1+\sum_{i=2}^{n-1}x_i\lambda_i+
2x_n\al_{1,n}$; 
\item[(b)]
if $x_1 < x_n$, then $x=\sum_{i=2}^{n-1} x_i \lambda_i+ (x_n - x_1) \lambda_n +
2 x_1\al_{1,n}$;
\item[(c)]
if $x_1 = x_n $, then $x=\sum_{i=2}^{n-1} x_i \lambda_i +
2 x_1\al_{1,n} = \sum_{i=2}^{n-1} x_i \lambda_i +
2 x_n\al_{1,n}$ ;
\end{itemize}
hence we may write $x$ as a positive linear combination of either
$\{\lambda_1, \ldots, \lambda_{n-1}, \al_{1,n}\}$, or  $\{\lambda_2, \ldots,
\lambda_{n}, \al_{1,n}\}$. Since, in general, any root $\alpha \in M_n$ is half
the sum of the unique long root in its row and the unique long root in its column, by
iterating this process we get the border of the desired  $I \in \mc
I_\ab(\al_n)$. Moreover, generically, the element $x=\sum_{i=1}^n x_i\lambda_i$
belongs to $\conv(B(I))$ for only one $I \in \mc I_\ab(\al_n)$: this does not
hold only if, at some step, it occurs that we are in case (c).
Hence 
$\vol\left(\conv(B(I))\cap
\conv(B(I'))\right)=0$. By Proposition \ref{unimodulareC}, for all $I \in \mc
I_\ab(\al_n)$, $\conv B(I) \cap \Phi = B(I)$ and therefore $\conv(B(I))\cap
\conv(B(I'))$ is a common face of $\conv(B(I))$ and $\conv(B(I'))$ for all 
$I, I' \in \mc I_\ab(\al_n)$.
\end{proof}

From the triangulation of $F$, we can construct a triangulation of
the whole $\mc P$  through the action of the elements in $W^n$.
Hence we obtain the following result.

\begin{thm}
\label{tri_c}
Let 
$$
\mc T_n=\{\conv_0(B(I))\mid I\in \mc I_\ab(\al_n)\}\quad\text{and}\quad
\mc T=W^n\mc T_n.
$$
Then $\mc T$ is a triangulation of  $\mc P$.
\end{thm}

\begin{rem}\label{numeroC}
The cardinality of $W^n$ is $2^n$ and the number of border strips in $\mc
T$ is $2^{n-1}$. Hence the total number of simplices in $\mc T$ is $2^{2n-1}$.
\end{rem}

Actually, this triangulation induces also a triangulation of $\mc P^+$. 
To prove this, we need the following lemma. Recall that, for any long positive
root
$\lambda$, we have set
\begin{itemize}
 \item[] $U_{\lambda}:= \{\be \in \Phi^+ \mid \be - \lambda \in \Phi^+\}$,
\item[] $R_{\lambda}:= \{\be \in \Phi^+ \mid \lambda - \be \in \Phi^+\}$
\end{itemize}
(see equations   (\ref{u-r}) and the subsequent figure).

Note that $U_{\lambda}$ and $M\cap R_{\lambda}$ both contain only short roots
$\be$ such that $2 \be - \lambda \in \Phi^+$. Equivalently, given any
positive root $\be \in M$, say $\be = \sum_{i=i_1}^{i_2}\al_i +
2 \sum_{i=i_2+1}^{n-1}\al_i + \al_n$, the hook centered at $\be$
contains the two long positive roots $\lambda_{i_1}= 2 \sum_{i=i_1}^{n-1}\al_i +
\al_n$ and $\lambda_{i_2} = 2 \sum_{i=i_2+1}^{n-1}\al_i +
\al_n$ satisfying $\lambda_{i_1} + \lambda_{i_2} = 2 \be$.

\begin{lem}
\label{sopra-destra}
Let $t \in [n]$, $w \in W^n$, and $\lambda$ be a positive long root. 
\begin{enumerate}
 \item 
\label{sopra-destra1}
If 
$(w (\lambda) , \breve\omega_t)<0$, then $(w(\be) , \breve\omega_t) \leq
0$ for all $\be \in U_{\lambda}$.
\item 
\label{sopra-destra2}
If  
$(w (\lambda) , \breve\omega_t)>0$, then $(w(\be) , \breve\omega_t) \geq
0$ for all $\be \in M\cap R_{\lambda}$.
\end{enumerate}
\end{lem}
\begin{proof}
Let us prove (\ref{sopra-destra1}). As we have noted above, $2 \be - \lambda
\in \Phi^+$. Since 
$$
(w (2\be - \lambda) , \breve\omega_t)= 2 (w (\be) , \breve\omega_t) -
(w (\lambda) , \breve\omega_t) \geq 2 (w (\be) , \breve\omega_t) +1,
$$
$(w (\be) , \breve\omega_t)$ cannot be positive otherwise $(w (2\be -
\lambda) , \breve\omega_t)$ would be $\geq 3$, which is impossible since
$m_t  \in \{1,2\}$ for type $C_n$.
\par

The proof of (\ref{sopra-destra2}) is analogous.
\end{proof}

\begin{thm}
\label{indotta_c}
Let
$\mathcal C^+$ denote the positive cone generated by $\Pi$. Then $\mathcal
P^+=\mathcal
C^+\cap
\mathcal P$
and the triangulation ${\mathcal T}$  of $\mathcal P$ restricts to a
triangulation of $\mathcal P^+$.
\end{thm}

\begin{proof}
As for the $A_n$ analogue (Theorem \ref{indotta}), we shall prove the following
two statements, which give the result:
\begin{enumerate}
 \item the triangulation ${\mathcal T}$  restricts
to a
triangulation of $\mathcal C^+\cap
\mathcal P$,
\item $\mathcal
P^+=\mathcal C^+\cap
\mathcal P$.
\end{enumerate}
The first assertion is equivalent to requiring that every simplex of the
triangulation $\mathcal T$ is either contained in $\mathcal C^+$, or intersects
it in a null set. 
Hence, we must show that, for any  $w\in W^n$, and $T\in \mathcal T_n$, either
$w(T)\subseteq \mathcal C^+$, or $w(T)\cap \mathcal C^+$ has volume equal to
$0$.
\par

So, let $w\in W^n$, $T\in\mathcal T_n$, $T=\conv_0 (B(I))$  with $I\in \mc
I_\ab(\al_n)$,
be such that $w(T)\not\subseteq \mc C^+$. Then,
$w(B(I))\not\subseteq \Phi^+$, hence there exists $\gamma\in B(I)$ and $t \in
[n]$ such that $(w(\gamma), \breve\omega_t)<0$. We need to prove that, for all
$\gamma'\in B(I)$, we have that $(w (\gamma') , \breve\omega_t)\leq 0$.

For all the long roots $\lambda$ belonging to columns which are weakly to the right of the column of $\gamma$, we have $\lambda \leq \gamma$, and hence 
 $(w(\lambda) , \breve\omega_t)<0$, by Proposition \ref{dispone}. Therefore, Lemma \ref{sopra-destra},
(\ref{sopra-destra1}), implies that $(w (\gamma') , \breve\omega_t)\leq 0$ for all $\gamma'\in M$ belonging to columns weakly to the right of the column of $\gamma$.

Now let  $\lambda$ be the long positive root in
the row of $\ga $. By Lemma \ref{sopra-destra}, (\ref{sopra-destra2}), 
$(w(\lambda) , \breve\omega_t)\leq 0$. Hence Proposition \ref{dispone}
 implies that $(w (\gamma') , \breve\omega_t)\leq 0$ for all $\gamma'\in B(I)$ belonging to columns to the left of the column of $\gamma$.

The proof that $\mc P^+=\mc  P\cap \mc C^+$ is similar to that of the $A_n$
analogue in Theorem \ref{indotta}.
\end{proof}

As in the $A_n$ case, we obtain, as a corollary of Theorem \ref {tri_c}, the
fact that   
the triangulation $\mc T'_n$ of Proposition \ref{tri-n} inherits the poset structure of $\mc I_\ab(\al_n)$. 

\begin{cor}
Let $I\in \mc I_\ab(\al_n)$ be an abelian ideal. Then the set 
$$
\mc T'_I=\{\conv(B(J))\mid J\in \mc I_\ab(\al_n),\ J\subseteq I\}
$$ 
is a triangulation of $\conv(I)$. 
\end{cor}

Note that our triangulation of type $C$ coincides with the triangulation
studied in \cite{K2} (in the usual coordinate description) on the cone on $F$, but not on
$\mc P^+$. In fact, we can check directly that our minimal paths from
$\al_{1, n}$ to a long root $\lambda_j$ ($j\in [n])$ correspond to the {\it
alternating well structured graphs of \cite{K2}} in which all edges are
positive. But, for example, in type $C_4$, $(\varepsilon_1 - \varepsilon_3, \varepsilon_2 -
\varepsilon_3, \varepsilon_2 + \varepsilon_4, 2 \varepsilon_4)$, which gives
rise to a graph which is not alternating well structured, is the image of 
$(\varepsilon_1 + \varepsilon_4, \varepsilon_2 +
\varepsilon_4, \varepsilon_2 + \varepsilon_3, 2 \varepsilon_3)$ under the action of
$w=s_{\alpha_3} s_{\alpha_4} \in W^4$.

\section{Some remarks on volumes}
\label{volume}
In this section, we show how the curious 
identity of \cite{DC-P} holds
also for the triangulations studied in Sections \ref{tr-a} and \ref{tr-c}.
\par

Let us set
$$
\vol(\Pi):=\vol(\conv_0(\Pi)).
$$   
If $B$ is any 
$\ganz$-basis of the root lattice,  then $\vol(\conv_0(B))=\vol(\Pi)$, since the linear transformation that maps
$\Pi$ to $B$ and its inverse are both integral and hence have determinant $1$ or
$-1$.
\par

By Propositions \ref{unimodulare} and \ref {unimodulareC}, it follows that
the total number of simplices in the  triangulations of Theorems \ref{tri} and
\ref{tri_c} is equal to 
$$
\frac{\vol(\mc P)}{\vol(\Pi)},
$$ 
where
$\vol({\mc P})$ is the volume of $\mc P$.  
Hence,
$\frac {\vol(\mc
P_{A_n})}{\vol(\Pi_{A_n})}=\binom{2n}{n}$ 
and
$\frac{\vol(\mc P_{C_n})}{\vol(\Pi_{C_n})}=2^{2n-1}.$
\par

By \cite[Theorem 3.11, Lemma 3.12, and Table 4]{FZ}, a
triangulation of $\mc P^+$ made of simplices generated by integral
bases has 
$$
\prod_{i=1}^n \frac{h+e_i-1}{e_i+1}
$$ 
elements, where $h$ is the Coxeter number and the $e_i$ are the exponents of
$\Phi$. For $\Phi$ of type $A_n$ this number specializes to the $n$th  Catalan
number $\frac{1}{n+1}\binom{2n}{n}$, and for type $C_n$ to $\binom{2n-1}{n}$.
For type $A_n$, we could already find this number in \cite{GGP}.    
\par

Thus we obtain that
$\frac{\vol(\mc P_{A_n}^+)} {\vol(\mc
P_{A_n})}=\frac{1}{n+1},$ and $\frac{\vol(\mc P_{C_n}^+
)}{\vol(\mc
P_{C_n})}=\frac{1}{2^{2n-1}}\binom{2n-1}{n}=\frac{1}{2^{2n}}\binom{2n}{n}.$
We can easily check that in both cases we have obtained that  
$$
\frac{\vol(\mc P^+)} {\vol(\mc
P)}=\frac{\prod_{i=1}^n e_i}{|W|}.
$$

This result is formally similar to the following one, first proved by De Concini
and Procesi \cite{DC-P},  that holds for all finite
crystallographic root systems (see also \cite{D}, \cite{BZ}, \cite{W}). 
Let $S$ be a sphere centered
at the origin. Then
$$
\frac{\vol(\mc C^+\cap S)} {\vol(S)}=\frac{\prod_{i=1}^n
e_i}{|W|},
$$ 
where, as before, $\mc C^+$ is the positive cone generated by $\Pi$.

\par

Indeed, we can see that the analogous equality holds for the orbit of each facet
of $\mc P$, in case $A_n$ and, trivially, in case $C_n$.
\par
 
Assume that $\Phi$
is of type $A_n$. For any $i\in [n]$, set 
$$
\wt F_i:=\conv_0(F_i).
$$ 
Consider the union of the sets in the orbit of $\wt F_i$ 
$$
W\wt F_i=\{w(x)\mid w\in W,\ x\in \wt F_i\}
$$ 
and the sets of simplices
$$
W^i \mc T_i=\{w(S)\mid w\in W^i,\ S\in \mc T_i\}.
$$ 
By Theorem \ref{indotta},  $W^i \mc T_i\cap \mc P^+$ is a
triangulation of 
$W \wt F_i \cap \mc P^+$. We have the following result.

\begin{pro}
Assume that $\Phi$ is of type $A_n$.
For each $i\in [n]$, 
$$
\frac {|W^i \mc T_i \cap \mc P^+|}{|W^i \mc
T_i|}=\frac{1}{n+1}.
$$
\end{pro} 

\begin{proof}
For any basis $\Pi'$ of the root system $\Phi$, let $\mc C_{\Pi'}^+$ denote the positive cone
generated by $\Pi'$. 
Since $W$ is a group of isometries and is
transitive on the set of root system bases of $\Phi$,  for any fixed standard  parabolic
facet $F_i$ we have that  
$$
\vol(W\tilde F_i \cap \mc C_{\Pi'}^+)=\vol(W\tilde F_i \cap \mc C^+)=
\vol(W\tilde F_i \cap \mc P^+)
$$
for all bases $\Pi'$ of $\Phi$. 
\par

It is easy to see that, for any finite crystallographic irreducible $\Phi$, 
$\gen \, \Phi$ is the union of the $n+1$ positive cones generated by
the sets $\Pi$ and $ \{-\theta\} \cup \Pi \setminus \{\al_k\}$, for all $k\in
[n]$, and that, moreover, these cones have pairwise null intersections (a proof 
can be found in \cite{DC-P}).

If $\Phi$ is of type $A_n$, the set $ \{-\theta\} \cup \Pi \setminus \{\al_k\}$
is a basis of $\Phi$, for  all $k\in
[n]$, and hence the volume of the intersection of $W\wt F_i$ with each of these cones is the same and is equal to $\frac{1}{n+1}\vol(W\wt F_i\cap \mc P)$, so that 
$$
\vol(W\wt F_i\cap \mc P^+)=\frac{1}{n+1}\vol(W\wt F_i\cap \mc P).
$$
Since all the simplices in $\mathcal T$ have the same volume,
and both $W\wt F_i\cap \mc P^+$  and $W\wt F_i\cap \mc P$ are union of simplices in $\mc T$, this proves the claim.
\end{proof}

\section*{Acknowledgments}
We are grateful to the referees for carefully reading our manuscript and for giving many helpful suggestions for improving the article. 
\par
Part of this work was done when the second author was at Universit\`a di Chieti-Pescara, whose support is gratefully acknowledged.

\end{document}